\documentclass{amsart}
\usepackage{amsmath,amsthm,amssymb}
\usepackage{dsfont}
\usepackage{mathrsfs} 
\usepackage[all,cmtip]{xy}
\input{diagxy}

\usepackage[cm]{fullpage}

 \theoremstyle{plain}
 \newtheorem{thm}{Theorem}[section]
 \newtheorem{cor}[thm]{Corollary}
 \newtheorem{lem}[thm]{Lemma}
 
 \newtheorem{prop}[thm]{Proposition}
 
\theoremstyle{definition}
 \newtheorem{defn}[thm]{Definition}

\theoremstyle{remark}
 \newtheorem{rem}[thm]{Remark}

 \newtheorem{conv}[thm]{Convention}
 \newtheorem{exam}[thm]{Example}

 \numberwithin{equation}{section}

\newtheorem*{incla}{Claim}
\newtheorem*{sta}{\textbf{\textup{Claim}} $(\star)$}

\DeclareMathOperator{\VF}{VF} \DeclareMathOperator{\ACVF}{ACVF}
 
\DeclareMathOperator{\RV}{RV} 
\DeclareMathOperator{\MM}{\mathcal{M}}
\DeclareMathOperator{\LL}{\mathcal{L}}
\DeclareMathOperator{\OO}{\mathcal{O}}

\DeclareMathOperator{\UU}{\mathcal{U}}

 \DeclareMathOperator{\ran}{ran}

 \DeclareMathOperator{\id}{id}

 \DeclareMathOperator{\Th}{Th}

 \DeclareMathOperator{\lh}{lh}

 \DeclareMathOperator{\cha}{char}
 \DeclareMathOperator{\ac}{\overline{ac}}

 \DeclareMathOperator{\acl}{acl}
 
 \DeclareMathOperator{\pr}{pr}

\DeclareMathOperator{\K}{\overline{K}}

\DeclareMathOperator{\rv}{rv}
\DeclareMathOperator{\vv}{val}

\DeclareMathOperator{\fib}{fib}

\DeclareMathOperator{\rad}{rad}
\DeclareMathOperator{\vcr}{vcr}
\DeclareMathOperator{\vrv}{vrv}
\DeclareMathOperator{\RVH}{RVH}

\DeclareMathOperator{\can}{\mathbf{c}}

\DeclareMathOperator{\sn}{sn}

\DeclareMathOperator{\dvn}{div}

\def\XXint#1#2#3{{\setbox0=\hbox{$#1{#2#3}{\int}$}
\vcenter{\hbox{$#2#3$}}\kern-.5\wd0}}

\newcommand{\Z}{\mathds{Z}}

\newcommand{\N}{\mathds{N}}

\newcommand{\p}{$p$\nobreakdash}

\newcommand{\bmin}{$b$\nobreakdash}
\newcommand{\cmin}{$C$\nobreakdash}

\newcommand{\gB}{\mathfrak{B}}
\newcommand{\gC}{\mathfrak{C}}

\newcommand{\ga}{\mathfrak{a}}
\newcommand{\gb}{\mathfrak{b}}
\newcommand{\gc}{\mathfrak{c}}
\newcommand{\gd}{\mathfrak{d}}

\newcommand{\gh}{\mathfrak{h}}

\newcommand{\go}{\mathfrak{o}}
\newcommand{\gp}{\mathfrak{p}}

\newcommand{\gs}{\mathfrak{s}}

\newcommand{\0}{\emptyset}

 \newcommand{\abs}[1]{\left\vert#1\right\vert}
 
 \newcommand{\set}[1]{\left\{#1\right\}}
 \newcommand{\seq}[1]{\left<#1\right>}
 \newcommand{\norm}[1]{\left\Vert#1\right\Vert}

 \newcommand{\lan}[1]{\mathcal{L}_{\textup{#1}}}

\newcommand{\limplies}{\rightarrow}

\newcommand{\fa}[1]{\forall #1 \;} 

\newcommand{\rest}{\upharpoonright}
\newcommand{\lbar}{\vec}

\newcommand{\fun}{\longrightarrow}
\newcommand{\efun}{\longmapsto}
\newcommand{\sub}{\subseteq}

\newcommand{\mi}{\smallsetminus}

\newcommand{\la}{\langle}
\newcommand{\ra}{\rangle}

\title[QE and minimality in ACVF]{Quantifier elimination and minimality conditions in algebraically closed valued fields}

\author[Y. Yin]{Yimu Yin}

\address{Department of Mathematics, University of Pittsburgh, 301 Thackeray Hall, Pittsburgh, PA 15260}

\email{yimuyin@pitt.edu}

\begin{document}

\begin{abstract}
A Basarab-Kuhlmann style language $\lan{RV}$ is introduced in the Hrushovski-Kazhdan integration theory~\cite{hrushovski:kazhdan:integration:vf}. The theory $\ACVF$ of algebraically closed valued fields formulated in this language admits quantifier elimination, which is not proved in~\cite{hrushovski:kazhdan:integration:vf} and the reader is referred to a result about a much more complicated language. In this paper, using well-known facts in the theory of valued fields, we give a straightforward proof. We also show that two expansions $\ACVF^{\dag}$ and $\ACVF^{\ddag}$ of $\ACVF$, one with a section of the entire $\RV$-sort and the other with a section of the residue field, admit quantifier elimination. Thereafter we show that, in terms of certain minimality conditions, the three theories are distinct geometrically.
\end{abstract}

\maketitle

\section{Introduction}

In this paper we study quantifier elimination (QE) for algebraically closed valued fields in a particular language $\lan{RV}$ and some of its geometrical consequences. The first QE result for algebraically closed valued fields is due to Robinson~\cite{Robinson1956}, where he used a one-sorted language $\lan{val}$ that uses a linear divisibility relation $\dvn$ to express the valuation. Later, Weispfenning~\cite{weispfenning:1983} gave a primitive recursive procedure of QE in the natural two-sorted language $\lan{v}$ for valued fields (one sort for the field and the other sort for the value group):

\begin{thm}[Weispfenning]\label{weis:qe}
The theory of algebraically closed valued fields as formulated in $\lan{v}$ admits QE.
\end{thm}

Another important QE result is due to Delon~\cite{F:delon:1981}, where she used a natural three-sorted language (a third sort for the residue field).

In the Hrushovski-Kazhdan integration theory~\cite{hrushovski:kazhdan:integration:vf} a Basarab-Kuhlmann style two-sorted language $\lan{RV}$ for algebraically closed valued fields is introduced, whose second sort $\RV$ is meant for the residue multiplicative structure. The corresponding theory is called $\ACVF(p,q)$, where $(p, q)$ indicates the characteristics of the field and the residue field. A basic motivation for the introduction of such a sort $\RV$ is to develop an integration theory for valued fields that are not equipped with an angular component map $\ac$. The map $\ac$ is a crucial ingredient in the Cluckers-Loeser integration theory~\cite{cluckers:loeser:constructible:motivic:functions}. This theory may be applied in general to the field of formal Laurent series over a field of characteristic $0$, but it heavily relies on the Cell Decomposition Theorem of
Denef-Pas~\cite{Denef:1986,Pa89}, which is only achieved for
valued fields of characteristic $0$ that are equipped with $\ac$. However, an angular component map is not
guaranteed to exist for just any valued field, for example,
algebraically closed valued fields. The Hrushovski-Kazhdan integration theory does not require the presence of $\ac$ and hence is of great foundational importance for the development of motivic integration.

To be more precise, it is not QE that is needed in~\cite{hrushovski:kazhdan:integration:vf}, but rather an important geometrical consequence of it, namely $C$-minimality (see~\cite{macpherson:steinhorn:variants:ominimal, haskell:macpherson:1994}). Two major $C$-minimal theories that are covered in~\cite{hrushovski:kazhdan:integration:vf} are $\ACVF(0,0)$ and its rigid analytic expansions. QE is still a fundamental tool in studying the models of these two theories. For them, the Hrushovski-Kazhdan integration theory may be simplified  through techniques that combine QE and $C$-minimality. For $\ACVF(0,0)$ this has been done in~\cite{yin:hk:part:1}.

In this paper we shall give a proof of QE for
$\ACVF$ of any characteristic. Note that this is not directly
proved in~\cite{hrushovski:kazhdan:integration:vf} and the reader
is referred to~\cite{haskell:hrushovski:macpherson:2006}. The
theme of the latter is elimination of imaginaries and the relevant
results use a much more complicated language than $\lan{RV}$,
which do not seem to imply QE for $\ACVF$ in a
straightforward fashion. Our proof, except some fundamental facts
in the theory of valued fields, is elementary and self-contained.

We can expand $\ACVF$ with a section of either the entire $\RV$-sort or just the residue field. The resulting theories are called $\ACVF^{\dag}$ and $\ACVF^{\ddag}$. Similar languages have been considered in~\cite{hru:kazh:2009}. In this paper we shall also prove that both $\ACVF^{\dag}$ and $\ACVF^{\ddag}$ admit QE.

On the other hand, geometrically, or more precisely, in terms of minimality conditions, the three theories can be distinguished from one another. First of all, QE implies that $\ACVF$ is $C$-minimal. The theories $\ACVF^{\dag}$ and $\ACVF^{\ddag}$ are obviously not \cmin-minimal. However, all three theories $\ACVF(0,0)$, $\ACVF^{\dag}(0,0)$, and $\ACVF^{\ddag}(0,0)$ are $b$-minimal in the sense of~\cite{cluckers:loeser:bminimality}. Finally we shall introduce a natural local version of \cmin-minimality, called local \cmin-minimality, and show that $\ACVF^{\ddag}$ satisfies it but $\ACVF^{\dag}$ does not.

We note that $b$-minimality of $\ACVF(0,0)$ is covered by~\cite[Theorem~7.2.6]{cluckers:loeser:bminimality}. Our proof presents in detail some essential aspects of the deep analysis of definable sets in $\ACVF(0,0)$ developed in the Hrushovski-Kazhdan integration theory~\cite{hrushovski:kazhdan:integration:vf}. This shall be continued in a sequel.

\section{Preliminaries}

Let us first introduce the Basarab-Kuhlmann style language
$\lan{RV}$ for valued fields. This style
first appeared in \cite{basarab:1991, basarab:kuhlmann:1992} and has been further investigated in
\cite{kuhlmann:1994, scanlon:thomas:QE:relative:Frobenius}. Its main feature is
the use of a countable collection of residue multiplicative
structures, which are reduced to just one for valued fields of
pure characteristic $0$.

\begin{defn}
The language $\lan{RV}$ has the following sorts and symbols:
\begin{enumerate}
 \item a $\VF$-sort, which uses the language of rings
 $\lan{R} = \set{0, 1, +, -, \times}$;
 \item an $\RV$-sort, which uses
  \begin{enumerate}
    \item the group language $\set{1, \times}$,
    \item two constant symbols $0$ and $\infty$,
    \item a unary predicate $\K^{\times}$,
    \item a binary function $+ : \K^2 \fun \K$ and a
    unary function $-: \K \fun \K$, where $\K =
    \K^{\times} \cup \set{0}$,
    \item a binary relation $\leq$;
    \end{enumerate}
  \item a function symbol $\rv$ from the $\VF$-sort into the $\RV$-sort.
\end{enumerate}
\end{defn}
Technically speaking, the constant $0$ and the functions $+$, $-$
in the $\RV$-sort should all be relations. Note that, for notational convenience, we do not use different symbols for $0$ and $1$, since which ones are
being referred to should always be clear in context. The two sorts without the zero elements are respectively denoted as $\VF^{\times}$ and $\RV$; $\RV \mi \set{\infty}$ is denoted as $\RV^{\times}$; and $\RV \cup \set{0}$ is denoted as $\RV_0$.

Let $M$ be an $\lan{RV}$-structure and $A$ a subset of $M$. The substructure generated by $A$ in $M$ is denoted as $\la A \ra$. A substructure $N \sub M$ is \emph{$\VF$-generated} if there is a subset $A \sub \VF(N)$ such that $N  = \la A \ra$.

Valued fields are naturally $\lan{RV}$-structures. Let $(K, \vv)$ be a
valued field and $\OO$, $\MM$, $\K$, $\Gamma$ the corresponding valuation
ring, maximal ideal, residue field, and value group. The sort $\RV$ is interpreted as $\RV(K) = K^{\times} / (1 + \MM)$ and the function $\rv$ is interpreted as the canonical quotient map $K^{\times} \fun \RV(K)$. For each $a \in K$, $\vv$ is constant on the subset $a + a\MM$ and hence there is a naturally induced map
$\vrv$ from $\RV(K)$ onto the value group $\Gamma$. The relation $\leq$ is then interpreted as the ordering given by $\vrv$ and the ordering of $\Gamma$. The situation
is illustrated in the following commutative diagram
\begin{equation*}
\bfig
 \square(0,0)/^{ (}->`->>`->>`^{ (}->/<600, 400>[\OO \mi \MM`K^{\times}`\K^{\times}`
\RV(K);`\text{quotient}`\rv`]
 \morphism(600,0)/->>/<600,0>[\RV(K)`\Gamma;\vrv]
 \morphism(600,400)/->>/<600,-400>[K^{\times}`\Gamma;\vv]
\efig
\end{equation*}
where the bottom sequence is exact. We see that $\K$ and $\Gamma$ are naturally wrapped together in this one sort $\RV(K)$. Note that the existence of an
angular component $\ac : K^{\times} \fun \K^{\times}$ is
equivalent to the existence of a group homomorphism from $\RV(K)$
onto $\K^{\times}$ in the diagram.

\begin{defn}
\emph{The theory $\ACVF$ of algebraically closed valued fields in $\lan{RV}$} states the following:
\begin{enumerate}
 \item $(\VF, 0, 1, + , -, \times)$ is an algebraically close field;

 \item $(\RV^{\times}, 1, \times)$ is a divisible abelian
 group, where multiplication $\times$ is augmented by $t
 \times 0 = 0$ for all $t \in \K$ and $t \times \infty =
 \infty$ for all $t \in \RV_0$;

 \item $(\K, 0, 1, +, -, \times)$ is an algebraically closed field;

 \item the relation $\leq$ is a preordering on $\RV$ with
 $\infty$ the top element and $\K^{\times}$ the equivalence
class of 1;

 \item the quotient $\RV / \K^{\times}$, denoted as
 $\Gamma \cup \set{\infty}$, is a divisible ordered abelian
group with a top element, where the ordering and the group
operation are induced by $\leq$ and $\times$, respectively,
and the quotient map $\RV \fun \Gamma \cup \set{\infty}$ is
denoted as $\vrv$;

 \item the function $\rv : \VF^{\times} \fun \RV^{\times}$
 is a surjective group homomorphism augmented by $\rv(0) =
\infty$ such that the composite function
\[
\vv = \vrv \circ \rv : \VF \fun \Gamma \cup \set{\infty}
\]
is a valuation with the valuation ring $\OO =
\rv^{-1}(\RV^{\geq 1})$ and its maximal ideal $\MM =
\rv^{-1}(\RV^{> 1})$, where
\begin{gather*}
\RV^{\geq 1} = \set{x \in \RV: 1 \leq x},\\
\RV^{> 1} = \set{x \in \RV: 1 < x}.
\end{gather*}
\end{enumerate}
\end{defn}

The set $\OO \mi \MM$ of units in the valuation ring is sometimes
denoted as $\UU$. In any model of $\ACVF$, the function $\rv \rest
\VF^{\times}$ may be identified with the quotient map
$\VF^{\times} \fun \VF^{\times} / (1 + \MM)$. Hence an $\RV$-sort
element $t$ may be understood as a coset of $(1 + \MM)$ and we may write $a \in t$ to mean $a \in \rv^{-1}(t)$.

Although we do not include the multiplicative inverse function in
the $\VF$-sort and the $\RV$-sort, we always assume that, without
loss of generality, $\VF(S)$ is a field and $\RV^{\times}(S)$ is a
group for a substructure $S$ of a model of $\ACVF$.

Besides analytic expansions, there are other expansions of $\lan{RV}$ that are of some interest. If we add an angular component map then the resulting language is in effect a three-sorted Denef-Pas language. As a relatively easy consequence of the deep analysis of definable sets of $\ACVF(0,0)$ in~\cite{hrushovski:kazhdan:integration:vf}, a new proof of the QE result in Denef-Pas language~\cite{Pa89} may be obtained through specialization. This proof will be presented in a sequel.

In certain developments of motivic integration theory it is desirable to prescribe $\VF$-sort representatives for elements in $\RV$, or at least for elements in the residue field (for example see~\cite{hru:kazh:2009}).

\begin{defn}
A function $\sn : \RV \fun \VF$ is a \emph{section} of $\RV$ if
\begin{enumerate}
  \item $\sn \rest \RV^{\times}$ is a homomorphism of multiplicative groups and $\sn(\infty) = 0$,
  \item $\sn(t) \in t$ for every $t \in \RV$,
  \item $\sn(\K^{\times}) \cup \set{0}$ is a subfield of $\OO$.
\end{enumerate}
Similarly,  $\sn$ is a \emph{section} of $\K$ if it is the restriction of a section of $\RV$ to $\K^{\times}$ augmented by $\sn(t) = 0$ for every $t \in \RV \mi \K^{\times}$.
\end{defn}
Many discrete valued fields are equipped with a natural section of $\RV$ and hence a section of $\K$, for example, any field of formal Laurent series. The expansion of $\lan{RV}$ with such a function symbol $\sn$ shall be denoted as $\lan{RV}^{\dag}$. The theory $\ACVF^{\dag}$ in $\lan{RV}^{\dag}$ says that, in addition to the axioms of $\ACVF$, the function $\sn$ is a section of $\RV$. Similarly the theory $\ACVF^{\ddag}$ in $\lan{RV}^{\dag}$ says that the function $\sn$ is a section of $\K$.

For \cmin-minimality, the reader is referred to~\cite{macpherson:steinhorn:variants:ominimal, haskell:macpherson:1994} for some basic results concerning this notion. However, in this paper we use a specialized version of the \cmin-minimality condition that is simpler than the original one.

\begin{defn}
A subset $\gb$ of $\VF$ is an \emph{open ball} if there is a
$\gamma \in \Gamma$ and a $b \in \gb$ such that $a \in \gb$ if and
only if $\vv(a - b) > \gamma$. It is a \emph{closed ball} if $a
\in \gb$ if and only if $\vv(a - b) \geq \gamma$. It is an
\emph{$\rv$-ball} if $\gb = \rv^{-1}(t)$ for some $t \in \RV$. The
value $\gamma$ is the \emph{radius} of $\gb$, which is denoted as
$\rad (\gb)$. Each point in $\VF$ is a closed ball of radius $\infty$ and $\VF$ is a clopen ball of radius $- \infty$.

If $\vv$ is constant on $\gb$ --- that is, $\gb$ is contained in an $\rv$-ball --- then $\vv(\gb)$ is the \emph{valuative center} of $\gb$; if $\vv$ is not constant on
$\gb$, that is, $0 \in \gb$, then the \emph{valuative center} of
$\gb$ is $\infty$. The valuative center of $\gb$ is denoted by
$\vcr(\gb)$.

A subset $\gp \sub \VF^n \times \RV^m$ is an (\emph{open, closed, $\rv$-}) \emph{polydisc} if it is of the form $(\prod_{i \leq n} \gb_i) \times \set{\lbar t}$, where each $\gb_i$ is an (open, closed, $\rv$-) ball and $\lbar t \in \RV^m$.
If $\gp$ is a polydisc then the \emph{radius} of $\gp$, denoted as $\rad(\gp)$,
is $\min \set{\rad(\gb_i) : i \leq n}$. The open and closed polydiscs centered at a sequence of elements $\lbar a = (a_1, \ldots, a_n) \in \VF^n$ with radii $\lbar \gamma = (\gamma_1, \ldots, \gamma_n) \in \Gamma^n$ are respectively denoted as $\go(\lbar a, \lbar \gamma)$ and $\gc(\lbar a, \lbar \gamma)$.
\end{defn}

\begin{defn}
Let $\LL$ be a language expanding $\lan{RV}$. Let $M$ be a structure of $\LL$ that satisfies the axioms for valued fields. We say that $M$ is \emph{\cmin-minimal} if every parametrically definable subset of $\VF(M)$ is a boolean combination of balls. An $\LL$-theory $T$ is \emph{\cmin-minimal} if every model of $T$ is \emph{\cmin-minimal}.
\end{defn}

For motivation and basic results concerning $b$-minimality, the reader should consult~\cite{cluckers:loeser:bminimality}. For convenience, here we describe what $b$-minimality means for valued fields considered as $\lan{RV}$-structures.

\begin{defn}\label{def:b:min}
Let $\LL$ be a language expanding $\lan{RV}$. Any sort of $\LL$ other than the $\VF$-sort is called an \emph{auxiliary sort} and any subset of a product of some auxiliary sorts is called an \emph{auxiliary subset}. Let $M$ be a structure of $\LL$ that satisfies the axioms for valued fields. We say that $M$ is \emph{$b$-minimal} if the following three conditions are satisfied for every set of parameters $S$, every $S$-definable subset $A$ of $\VF(M)$, and every $A$-definable function $f : A \fun \VF(M)$.
\begin{itemize}
  \item[(b1)] There exists an $A$-definable function $P : A \fun U$ with $U$ auxiliary such that for each $t \in U$ the fiber $P^{-1}(t)$ is a point or a ball.
  \item[(b2)] If $g$ is a definable function from an auxiliary subset to a ball of radius $< \infty$ then $g$ is not surjective.
  \item[(b3)] There exists an $A$-definable function $P : A \fun U$ with $U$ auxiliary such that for each $t \in U$ the restriction $f \rest P^{-1}(t)$ is either injective or constant.
\end{itemize}
An $\LL$-theory $T$ is \emph{$b$-minimal} if every model of $T$ is $b$-minimal.
\end{defn}

\section{Quantifier elimination}\label{section:qe}

In this section we shall use Shoenfield's test~\cite{Shoen71} to show that $\ACVF$, $\ACVF^{\dag}$, and $\ACVF^{\ddag}$ all admit QE. Our strategy is to reduce the task to a case where we may apply Theorem~\ref{weis:qe}. This is based on the following simple observation.

\begin{rem}\label{L:v:and:L:RV}
With the imaginary $\Gamma$-sort and the valuation map $\vv$, $\lan{RV}$ may be viewed as an expansion of $\lan{v}$. Henceforth we shall refer to the two sorts of $\lan{v}$ as the $\VF$-sort and the $\Gamma$-sort. Under the natural interpretations, each valued field may be turned into an $\lan{RV}$-structure and an $\lan{v}$-structure. In fact, two valued fields are monomorphic as $\lan{RV}$-structures if and only if they are monomorphic as $\lan{v}$-structures. Let $K$, $L$ be two valued fields and $f' : K' \fun L'$ an $\lan{RV}$-isomorphism of two $\lan{RV}$-substructures of $K$, $L$. If $K'$ is $\VF$-generated then $f'$ may also be treated as an $\lan{v}$-isomorphism. If $L$ is sufficiently saturated then, by Theorem~\ref{weis:qe}, $f'$ may be extended to an $\lan{v}$-monomorphism $f : K \fun L$, which is also an $\lan{RV}$-monomorphism. Note that this procedure may fail if $K'$ is not $\VF$-generated, that is, if $\rv(\VF(K')) \neq \RV(K')$.

Of course everything said above is also true if $\lan{v}$ is replaced with Robinson's language $\lan{val}$.
\end{rem}

\begin{lem}\label{rv:root:lift:vf}
Let $B \sub M \models \ACVF$, $b_i \in \VF(B)$, and $F(X) = \sum_{0 \leq i \leq n} b_iX^i$. Let $t \in \RV(M)$ and $\lbar F(t) = \sum_{0 \leq i \leq n} t_it^i$ be a nonzero polynomial with coefficients in $\RV^{\times} \cup \set{0}$ such that $t_i = \rv(b_i)$ whenever $t_i \neq 0$. If $\lbar F(t) = 0$ and $\vrv(\rv(b_i)t^i) > 0$ for all $t_i = 0$, then there is a $b \in t$ such
that $F(b) = 0$.
\end{lem}
\begin{proof}
Without loss of generality we may assume $b_n \neq 0$. Let $F^*(X) = \sum_{t_i \neq 0} b_iX^i$. Fix a $t \in \RV(M)$ with $\lbar F(t) = 0$ and $\vrv(\rv(b_i)t^i) > 0$ for all $t_i = 0$. Note that, since such a $t$ exists and $\lbar F(X)$ is not the zero polynomial, we must have that $\lbar F(X)$ is not a monomial and $t \neq \infty$. This means that, for every $t_i \neq 0$, $\vrv(t_i t^i) = 0$. Let $m$ be the least number such that $t_m \neq 0$ and $l$ the greatest number such that $t_l \neq 0$. Fix a $b \in t$. Since $\vv(b_mb^m) = \vv(b_lb^l) = 0$ and $\vv(b_ib^i) > 0$ for all $t_i \neq 0$, $(m, \vv(b_m / b_n))$ and $(l, \vv(b_l / b_n))$ must be two adjacent vertices of the Newton polygon of $F(X)/ b_n$. Let $r_1, \ldots, r_n \in \VF(M)$ be the (possibly repeated) roots of $F(X)/ b_n$. For any $b \in t$, if $\rv(b) \neq \rv(r_i)$ for every $i$ then
\[
\begin{cases}
 \vv(b - r_i) = \vv(b),        &\text{ if } \vv(b) < \vv(r_i);\\
 \vv(b - r_i) = \vv(r_i),      &\text{ if } \vv(b) \geq \vv(r_i).
\end{cases}
\]
By the basic properties of Newton polygons, we have
\[
\sum_i \vv(b - r_i) = m \vv(b) + \sum_{\vv(r_i) \leq \vv(b)} \vv(r_i) = \vv(b_mb^m/b_n)
\]
and hence $\vv(F(b)) = \vv(b_mb^m) = 0$. So $\vv(F^*(b)) = 0$, contradicting the
choice of $t$. So $t = \rv(b) = \rv(r_i)$ for some $i$.
\end{proof}

\subsection{QE in $\ACVF$}
Now we fix two models $M$, $N \models \ACVF$ such that $N$ is $\norm{M}^+$-saturated. We shall work with a fixed substructure $S \sub M$ and a fixed monomorphism $f: S \fun N$.
%

\begin{lem}
There is a monomorphism $f^*: S^* \fun N$ extending $f$ such that
\begin{enumerate}
  \item $\VF(S^*) = \VF(S)$,
  \item $\K(S^*) = \K(M)$,
  \item $\Gamma(S^*)$ is the divisible hull of $\Gamma(S)$.
\end{enumerate}
\end{lem}
\begin{proof}
First of all, there is a field homomorphism $g: \K(M) \fun \K(N)$ extending $f \rest \K(S)$. Let $S_1$ be the substructure $\seq{\K(M), \RV(S)}$. Let $f_1 : S_1 \fun N$ be the monomorphism determined by
\[
ts \efun g(t)f(s) \text{ for all } t \in \K(M) \text{ and } s \in \RV(S).
\]
Next, let $n > 1$ be the least natural number such that there is a $t
\in \RV(M)$ with $t^n \in S_1$ but $t^i \notin S_1$ for every $0 < i < n$. Let $r \in \RV(N)$ such that $f_1(t^n) = r^n$. Let $f_2 : \seq{S_1, t} \fun N$ be the monomorphism determined by
\[
ts \efun r f_1(s) \text{ for all } s \in S_1.
\]
Iterating this procedure the lemma follows.
\end{proof}

By this lemma, we may and shall assume that $\K(S) = \K(M)$ and $\Gamma(S)$ is divisible.

Let $\hat S = \seq{\VF(S)}$. Fix an $e \in \VF(M)$ such that $\rv(e) \in \RV(S) \mi \RV(\hat S)$. In the next few lemmas, under various assumptions, we
shall prove the following claim:
\begin{sta}
$\RV(\la \hat S, e \ra) = \la \RV(\hat S), \rv(e) \ra \sub \RV(S)$ and the monomorphism $f \rest \hat S$ may be extended to another monomorphism $f^* : \la \hat S, e \ra \fun N$ such that $f^*(\rv(e)) = f(\rv(e))$.
\end{sta}

Note that Claim~$(\star)$ immediately implies $f^* \rest (\la \hat S, e \ra \cap S) \sub f$.

\begin{lem}\label{lift:all:root}
Let $F(X) = X^n + \sum_{0 \leq i < n}a_{i}X^{i} \in \OO(\hat S)[X]$ such that its projection $\lbar F(X)$ to $\K(\hat S)[X]$ is an irreducible polynomial. Suppose that $e \in \UU(M)$ and is a root of $F(X)$. If the valued
field $(\VF(\hat S), \OO(\hat S))$ is henselian, then Claim~$(\star)$ holds.
\end{lem}
\begin{proof}
Since $\rv(e)$ is a root of $\lbar F(X)$, $f(\rv(e))$ is a root of the irreducible polynomial $f(\lbar F(X))$. By Lemma~\ref{rv:root:lift:vf}, there is a root $d \in \VF(N)$ of $f(F(X))$ such that $\rv(d) = f(\rv(e))$. Since $F(X)$, $f(F(X))$ are irreducible over $\VF(\hat S)$, $f(\VF(\hat S))$, respectively, and $(\VF(\hat S), \OO(\hat S))$ is henselian, there is a valued field embedding $f^* : \la \hat S, e \ra \fun N$
with $f^*(e) = d$ that extends the valued field embedding $f$. By Remark~\ref{L:v:and:L:RV}, $f^*$ may be naturally converted into an $\lan{RV}$-monomorphism that extends the $\lan{RV}$-monomorphism $f$.

Now, by the fundamental inequality of valuation theory (see~\cite[Theorem~3.3.4]{engler:prestel:2005}), we have
\[
[\K(\la \hat S,e \ra): \K(\hat S)] = [\VF(\la \hat S,e \ra):\VF(\hat S)],
\]
and hence
\[
\K(\la \hat S,e \ra) = \K(\hat S)(\rv(e)), \quad
\Gamma(\la \hat S,e \ra) = \Gamma(\hat S).
\]
Therefore $\RV(\la \hat S, e \ra) = \la \RV(\hat S), \rv(e) \ra$.
\end{proof}

\begin{lem}\label{lift:divisible:hull}
Suppose that $e \notin \UU(M)$, $e^n = a \in \VF(\hat S)$ for some
integer $n > 1$, and $\vv(e^i) \notin \Gamma(\hat S)$ for all $0< i <
n$. If $(\VF(\hat S), \OO(\hat S))$ is henselian, then Claim~$(\star)$
holds.
\end{lem}
\begin{proof}

Any element $b \in \VF(\la \hat S,e \ra)$ may be written as a quotient of
two elements of the form $\sum_{0 \leq i \leq m} b_ie^i$, where
$b_{i} \in \VF(\hat S)$. Since $e^n = a \in \VF(\hat S)$, we may assume $0 \leq m < n$. For any $i < j$, if $b_i$ and $b_j$ are nonzero then $\vv(b_ie^i) \neq \vv(b_je^j)$, because otherwise we would have $\vv(e^{j - i}) \in \Gamma(\hat S)$. So
\[
\rv \biggl( \sum_{0 \leq i \leq m} b_ie^i \biggr) = \rv(b_je^j)
\]
for some $j$. So $\RV(\la \hat S, e \ra) = \la \RV(\hat S), \rv(e) \ra$.


Note that, since the roots of the polynomial $X^n - a$ are all of the same
value, by the assumption on $\vv(e)$, $F(X)$ is irreducible over
$\VF(\hat S)$. Since
\[
\frac{f(\rv(e))^n}{\rv(f(a))} - 1 = 0,
\]
by  Lemma~\ref{rv:root:lift:vf}, there is a root $d \in \VF(N)$ of the polynomial $X^n / f(a) - 1$ such that $\rv(d) = f(\rv(e))$. Now we may proceed exactly as in the previous lemma.
\end{proof}

\begin{lem}\label{lift:transcendental}
Suppose that $e \in \UU(M)$ and $\rv(e)$ is transcendental over $\K(\hat S)$.
If $\Gamma(\hat S)$ is divisible, then Claim~$(\star)$ holds.
\end{lem}
\begin{proof}
Clearly $\rv(e)$ does not contain any element that is algebraic
over $\VF(\hat S)$; in particular, $e$ is transcendental over $\VF(\hat S)$.
Similarly $f(\rv(e))$ does not contain any element that is algebraic over $f(\VF(\hat S))$. Choose a $d \in \VF(N)$ with $\rv(d) = f(\rv(e))$.

By the dimension inequality of valuation theory
(see~\cite[Theorem~3.4.3]{engler:prestel:2005}), the rational rank
of $\Gamma(\la \hat S, e \ra) /\Gamma(\hat S)$ is 0. Since $\Gamma(\hat S)$ is
divisible, we actually have $\Gamma(\la \hat S, e \ra) = \Gamma(\hat S)$.
So for every $b \in \VF(\la \hat S, e \ra)$ there is an $a \in \VF(\hat S)$
such that $\vv(b / a) = 0$. Let
\[
b = \sum_{0 \leq i \leq m} b_ie^i \in \VF(\la \hat S, e \ra), \quad
b^* = \sum_{0 \leq i \leq m} f(b_i)d^i \in \VF(\la f(\hat S),d \ra),
\]
where $b_i \in \VF(\hat S)$.

\begin{incla}
If $\vv(b) = 0$ then
\begin{enumerate}
 \item $\rv(b) \in \K(\hat S)[\rv(e)]$ and $\rv(b^*) \in \K(f(\hat S))[\rv(d)]$,
 \item $\vv(b^*) = 0$.
\end{enumerate}
\end{incla}
\begin{proof}
We do induction on $m$. Without loss of generality we may assume $b_m \neq 0$, $b_0 \neq 0$, and $\vv(b_i) \leq 0$ for all $i$. First suppose that $\vv(b_0)
\neq \vv(e \sum_{j=1}^{m} b_je^{j-1})$. Then $\vv(b) = \vv(b_0) = 0$ and $\vv(\sum_{j=1}^{m} b_je^{j-1}) > 0$. Let $a \in \VF(\hat S)$ be such that $\vv(a) =
\vv(\sum_{j=1}^{m} b_je^{j-1})$. By the inductive hypothesis, $\vv(\sum_{j=1}^{m} f(b_j/a) d^{j-1}) = 0$ and hence $\vv(d \sum_{j=1}^{m} f(b_j) d^{j-1}) > 0$. So $\vv(b^*) = \vv(f(b_0)) = 0$ and $\rv(b^*) = \rv(f(b_0)) \in \K(f(\hat S))[\rv(d)]$.

Next suppose that $\vv(b_0) = \vv(e \sum_{j=1}^{m} b_j e^{j-1}) < 0$. Then, since $\vv(b / b_0) > 0$, we have $\rv(e) \rv ( \sum_{j=1}^{m} b_j e^{j-1}/b_0 ) + 1 = 0$. By the inductive hypothesis,
\[
\rv \biggl( \sum_{j=1}^{m} b_j e^{j-1}/b_0 \biggr) \in \K(\hat S)[\rv(e)].
\]
So the equality implies that $\rv(e)$ is algebraic over $\K(\hat S)$, contradiction.

Finally suppose that $\vv(b_0) = \vv(e \sum_{j=1}^{m} b_j e^{j-1}) = 0$.
In this case, by the inductive hypothesis, we have
\begin{gather*}
\rv(b) = \rv(e)\rv \biggl(\sum_{j=1}^{m} b_je^{j-1} \biggr ) + \rv(b_0) \in \K(\hat S)[\rv(e)]\\
\vv \biggl( \sum_{j=1}^{m} f(b_j) d^{j-1} \biggr) = 0, \quad
\rv \biggl(\sum_{j=1}^{m} f(b_j) d^{j-1} \biggr) \in \K(f(\hat S))[\rv(d)].
\end{gather*}
If $\vv(b^*) > 0$ then $\rv(d) \rv (\sum_{j=1}^{m} f(b_j) d^{j-1} ) + \rv(f(b_0)) = 0$ and hence $\rv(d)$ is algebraic over $\K(f(\hat S))$, contradiction. So
$\vv(b^*) = 0$ and
\[
\rv(b^*) = \rv(d) \rv \biggl(\sum_{j=1}^{m} f(b_j) d^{j-1} \biggr) +
\rv(f(b_0)) \in \K(f(\hat S))[\rv(d)],
\]
as required.
\end{proof}

Note that, symmetrically, the claim still holds if $b$ and $b^*$ are interchanged. It follows that the embedding of the field $\VF(\la \hat S, e \ra)$ into the field $\VF(N)$ determined by $e \efun d$ induces a valued field embedding $f^* : \la \hat S, e \ra \fun N$ that extends the valued field embedding $f$, which, again by Remark~\ref{L:v:and:L:RV}, may be naturally converted into an $\lan{RV}$-monomorphism. Clearly we also have $\RV(\la \hat S, e \ra) = \la \RV(\hat S), \rv(e) \ra$.
\end{proof}

\begin{lem}\label{e:transc}
Suppose that $e$ is transcendental over $\VF(\hat S)$ and $\vv(e)$ is
of infinite order modulo $\Gamma(\hat S)$. For any $b = \sum_{0
\leq i \leq m} b_ie^i$ with $b_i \in \VF(\hat S)$,
if $b \neq 0$ then $\vv(b) = \min \set{\vv(b_ie^i) : 0 \leq i \leq
m}$. Also, $\Gamma(\la \hat S, e \ra)$ is the direct sum of $\Gamma(\hat S)$
and the cyclic group generated by $\vv(e)$: $\Gamma(\la \hat S, e \ra) =
\Gamma(\hat S) \oplus (\Z \cdot \vv(e))$.
\end{lem}
\begin{proof}
This is well-known; see, for example, \cite[Lemma~4.8]{PrRo84}.
\end{proof}

\begin{lem}\label{lift:value:full}
If $\K(\hat S) = \K(M)$ and $\Gamma(\hat S)$ is divisible, then Claim~$(\star)$ holds.
\end{lem}
\begin{proof}
Since $\Gamma(\hat S)$ is divisible, clearly $\vv(e)$ is of infinite order
modulo $\Gamma(\hat S)$ and hence $e$ is transcendental over $\VF(\hat S)$.
Choose a $d \in \VF(N)$ with $\rv(d) = f(\rv(e))$. Then $d$ is
transcendental over $f(\VF(\hat S))$. As above, by
Lemma~\ref{e:transc} and Remark~\ref{L:v:and:L:RV}, the embedding of the field $\VF(\la \hat S, e \ra)$
into the field $\VF(N)$ determined by $e \efun d$ induces an $\lan{RV}$-monomorphism $f^* : \la \hat S, e \ra \fun N$ that extends $f$. Moreover, since $\K(\la \hat S, e \ra) = \K(M)$ and $\Gamma(\la \hat S, e \ra) = \la \Gamma(\hat S), \vv(e) \ra$, we clearly have $\RV(\la \hat S, e \ra) = \la \RV(\hat S), \rv(e) \ra$.
\end{proof}

\begin{prop}\label{lrv:mono:lift}
There is a monomorphism $f^* : M \fun N$ extending $f$.
\end{prop}
\begin{proof}
First of all, since the henselization $\hat S^h$ of $\hat S$ in $M$ is an immediate extension (in the sense of valuation theory), we have $\RV(\la \hat S^h, \hat S \ra) = \RV(\hat S)$. So we may assume that $(\VF(\hat S), \OO(\hat S))$ is henselian. Now we use Lemma~\ref{lift:all:root} to extend $f
\rest \hat S$ to $f_1 : \hat S_1 \fun N$ by adding all the elements in $\K(M)$ that are algebraic over $\K(\hat S)$. Manifestly $\K(\hat S_1)$ is algebraically closed. Then, starting with the least $n$ such that there is a $\gamma \in \Gamma(\hat S_1)$ that is not divisible by $n$, we use Lemma~\ref{lift:divisible:hull} to extend $f_1$ to $f_2 : \hat S_2 \fun N$ such that $\Gamma(\hat S_2)$ is divisible. Note that, by
the proof of Lemma~\ref{lift:divisible:hull}, $\K(\hat S_2) = \K(\hat S_1)$.
Next, we use Lemma~\ref{lift:transcendental} to extend $f_2$ to $f_3 : \hat S_3 \fun N$ by adding an element in $\K(M)$ that is transcendental over $\K(\hat S_2)$. Iterating these procedures we may exhaust all elements in $\K(M)$ and obtain a monomorphism $f_4 : \hat S_4 \fun N$ such that $\hat S_4$ satisfies the assumption of Lemma~\ref{lift:value:full}. Then, a combined application of
henzelization, Lemma~\ref{lift:divisible:hull}, and
Lemma~\ref{lift:value:full} eventually brings a monomorphism $f_5 : \hat S_5 \fun N$ extending $f$ such that $\hat S_5$ is $\VF$-generated. Now the proposition follows from Remark~\ref{L:v:and:L:RV}.
\end{proof}

This proposition and Shoenfield's test immediately yield:

\begin{thm}\label{qe:acvf}
The theory $\ACVF$ admits quantifier elimination.
\end{thm}

\begin{rem}
Converse QE holds in the following sense. Let $K$ be a valued field interpreted naturally as an $\lan{RV}$-structure. If $\Th(K)$ in $\lan{RV}$ admits QE then $K$ is algebraically closed. This follows easily from the argument in~\cite[Section~4]{MMV83}. To see it, as in~\cite{MMV83}, let $\lan{val}$ be Robinson's one-sorted language for valued fields. Observe that any $\lan{val}$-formula may be translated into an $\lan{RV}$-formula containing only $\VF$-sort parameters and any quantifier-free $\lan{RV}$-formula containing only $\VF$-sort parameters may be translated into a quantifier-free $\lan{val}$-formula. So $\Th(K)$ in $\lan{val}$ also admits QE.
\end{rem}

\subsection{QE in $\ACVF^{\dag}$ and $\ACVF^{\ddag}$}

Next we show that $\ACVF^{\dag}$ also admits QE. Let $M$, $N \models \ACVF^{\dag}$ such that $N$ is $\norm{M}^+$-saturated. Let $S$ be a substructure of $M$ and $f: S \fun N$ a monomorphism. Note that any substructure of a model of $\ACVF^{\dag}$ is $\VF$-generated.

\begin{lem}\label{ac:lift:acvfi}
Let $S^{ac} \sub M$ be the substructure generated by the field-theoretic algebraic closure of $\VF(S)$ in $M$. Then there is a monomorphism $f^* : S^{ac} \fun N$ extending $f$.
\end{lem}
\begin{proof}
Let $\dot{S}$, $\dot M$, $\dot N$, $\dot f$, and $\dot S^{ac}$ be the $\lan{RV}$-reducts of $S$, $M$, $N$, $f$, and $S^{ac}$. From general valuation theory we have that $\K(\dot S^{ac})$ is the field-theoretic algebraic closure of $\K(\dot S)$ and $\Gamma(\dot S^{ac})$ is the divisible hull of $\Gamma(\dot S)$. By Proposition~\ref{lrv:mono:lift}, there is an $\lan{RV}$-monomorphism $\dot f^* : \dot S^{ac} \fun \dot N$ extending $\dot{f}$. Let
\[
P = \{ \sn(s) : s^n = t \text{ for some } t \in \RV(\dot S) \text{ and some natural number } n \}.
\]
Note that $P$ is the set of all $n$th roots of elements in $\sn(\RV(\dot S))$. Hence $P$ is a subset of $\VF(\dot S^{ac})$ and $\dot f^*(\sn(s)) = \sn(\dot f^*(s))$ if $\sn(s) \in P$. Let $t \in \RV(\dot S^{ac})$ and $a = \sn(t)$.
If $t \in \K(\dot S^{ac})$ then there is a polynomial $\lbar F(X) = \sum_{i=1}^{n} t_i X^i$ with $t_i \in \K(\dot S)$ such that $\lbar F(t) = 0$. Let $a_i = \sn(t_i)$ if $t_i \neq 0$, otherwise set $a_i = 0$. Clearly $\sum_{i=1}^{n} a_i a^i = 0$ and hence $a \in  \VF(\dot S^{ac})$. So $\dot f^*(a) = \sn(\dot f^*(t))$. If $t \notin \K(\dot S^{ac})$ then there is a $b \in \sn(\K(\dot S^{ac}))$ and a $c \in P$ such that $a = bc$. So $a \in  \VF(\dot S^{ac})$ and
\[
\dot f^*(a) = \dot f^*(b) \dot f^*(c) = \sn(\dot f^*(\rv(b))) \sn(\dot f^*(\rv(c))) = \sn(\dot f^*(\rv(a))).
\]
Therefore $\dot f^*$ induces an $\lan{RV}^{\dag}$-monomorphism.
\end{proof}

\begin{lem}\label{lift:any:sec}
For any $e \in \sn(\RV(M)) \mi S$ there is a monomorphism $f^* : \la S, e\ra \fun N$ extending $f$.
\end{lem}
\begin{proof}
By Lemma~\ref{ac:lift:acvfi}, without loss of generality, we may assume that $\VF(S)$ is algebraically closed and hence $e$ is transcendental over $\VF(S)$. If $e \in \K(M)$ then we may apply Lemma~\ref{lift:transcendental} with $d = \sn(f(\rv(e)))$. Since $\RV(\la S, e \ra) = \la \RV(S), \rv(e) \ra$, the resulting map is evidently an $\lan{RV}^{\dag}$-monomorphism.

If $e \notin \K(M)$ then we choose a $t \in \RV(N)$ that makes the same Dedekind cut in $\Gamma(f(S))$ as $\rv(e)$ in $\Gamma(S)$. This is possible since $N$ is sufficiently saturated. Now we see that the proof of Lemma~\ref{lift:value:full} goes through with $d = \sn(t)$ and as above the resulting map is an $\lan{RV}^{\dag}$-monomorphism.
\end{proof}

\begin{thm}\label{qe:acvfi}
The theory $\ACVF^{\dag}$ admits quantifier elimination.
\end{thm}
\begin{proof}
By Lemma~\ref{lift:any:sec} there is a monomorphism $f_1 : S_1 \fun N$ extending  $f$ such that $\sn(\RV(M)) \sub \VF(S_1)$. At this point, any $\lan{RV}$-extension of $f_1$ is an $\lan{RV}^{\dag}$-extension of $f_1$. So QE follows from Proposition~\ref{lrv:mono:lift} and Shoenfield's test.
\end{proof}

It is easy to see that a simpler version of the proof of Theorem~\ref{qe:acvfi} works for $\ACVF^{\ddag}$ and hence we have:

\begin{thm}\label{qe:acvfii}
The theory $\ACVF^{\ddag}$ admits quantifier elimination.
\end{thm}

\section{Minimality in $\ACVF$}

In this section we shall establish \cmin-minimality (in a sense simpler than the original one in~\cite{macpherson:steinhorn:variants:ominimal, haskell:macpherson:1994}) for $\ACVF$ and $b$-minimality in the sense of~\cite{cluckers:loeser:bminimality} for $\ACVF(0,0)$ ($\ACVF$ of pure characteristic $0$).
The former follows quite easily from QE. The latter needs some analysis that needs \cmin-minimality.

\subsection{$C$-minimality and some basic structural properties}

Let $\gC$ be a sufficiently saturated model of $\ACVF$.
Fix a small substructure $S \sub \gC$ and let $\ACVF_S$ be the theory that extends $\ACVF$ with the atomic diagram of $S$. Hence $\ACVF_S$ is complete. We shall work in $\ACVF_S$. For notational simplicity we shall still refer to the language of $\ACVF_S$ as $\lan{RV}$. \emph{By a definable subset of $\gC$ we mean a $\0$-definable
subset in $\ACVF_S$}. If additional parameters are used in defining a subset then we shall spell them out explicitly if necessary.

\begin{defn}
Let $\lbar X$ be $\VF$-sort variables and $\lbar Y$ be $\RV$-sort variables.

A \emph{$\VF$-literal} is an $\lan{RV}$-formula of the form $F(\lbar X) \, \Box \, 0$, where $F(\lbar X)$ is a polynomial with coefficients in $\VF$, and $\Box$ is either $=$ or $\neq$.

A \emph{$\K$-term} is an $\lan{RV}$-term of the form $\sum_{i = 1}^k (\rv(F_{i}(\lbar X)) \cdot r_{i} \cdot \lbar Y^{n_i})$ with $k > 1$, where $F_{i}(\lbar X)$ is a polynomial with coefficients in $\VF$ and $r_{i} \in
\RV$. An \emph{$\RV$-literal} is an $\lan{RV}$-formula of the form
\[
\rv(F(\lbar X)) \cdot \lbar Y^m \cdot T(\lbar X, \lbar Y)
\, \Box \, \rv(G(\lbar X)) \cdot r \cdot \lbar Y^l \cdot S(\lbar X, \lbar Y),
\]
where $F(\lbar X)$, $G(\lbar X)$ are polynomials with coefficients in $\VF$, $T(\lbar X, \lbar Y)$, $S(\lbar X, \lbar Y)$ are $\K$-terms, $r \in \RV$, and $\Box$ is one of the symbols $=$, $\neq$,
$\leq$, and $>$.
\end{defn}

Note that if $T(\lbar X, \lbar Y)$ is a $\K$-term, $\lbar a \in \VF$, and $\lbar t \in \RV$ then $T(\lbar a, \lbar t)$ is defined if and only if each summand in $T(\lbar a, \lbar t)$ is either of value $1$ or is equal to $0$. Also, since the value of $\K$-terms are $0$, we may assume that they do not occur in $\RV$-sort inequalities.

Any $\lan{RV}$-formula with parameters is provably equivalent to a disjunction of conjunctions of $\VF$-literals and $\RV$-literals. This follows from Theorem~\ref{qe:acvf} and routine syntactical inductions.

\begin{thm}\label{c:min:acvf}
The theory $\ACVF$ is $C$-minimal.
\end{thm}
\begin{proof}
Let $X$ be a $\VF$-sort variable and $\phi(X)$ a quantifier-free $\lan{RV}$-formula with parameters, where $X$ is the only variable in $\phi(X)$. By introducing more $\VF$-sort parameters, across a disjunction, any $\K$-term in $\phi(X)$ is reduced to either $0$ or the form $\rv(F(X))$. Note that in any $\RV$-literal, according to the syntax, if one side of $\Box$ is $0$ then the other side must be a $\K$-term and $\Box$ is either $=$ or $\neq$. Hence any $\RV$-literal in $\phi(X)$ is reduced to one of the following two forms: $0 \, \Box \, T(X)$ and $\rv(F(X)) \, \Box \, \rv(G(X))$. So the subset defined by $\phi(X)$ is also definable by an $\lan{v}$-formula and \cmin-minimality follows from~\cite[Theorem~4.11]{macpherson:steinhorn:variants:ominimal}.
\end{proof}

For any small subset $A \sub \gC$ let $\acl(A)$ be the model-theoretic algebraic closure of $A$ in $\gC$.

\begin{lem}\label{exchange}
The exchange principle holds in both sorts:
\begin{enumerate}
 \item For any $a$, $b \in \VF$, if $a \in \acl(b) \mi \acl(\0)$ then $b \in \acl(a)$.
 \item For any $t$, $s \in \RV$, if $t \in \acl(s) \mi \acl(\0)$ then $s \in \acl(t)$.
\end{enumerate}
\end{lem}
\begin{proof}
For the first claim, let $\phi(X, b)$ be a quantifier-free formula in
disjunctive normal form that witnesses $a \in \acl(b)$. Let $F(X, b)$ be a polynomial occurring in $\phi(X, b)$. If $F(X, b) = 0$ then, since $a \notin \acl(\0)$, some coefficient of $F(X, b)$ is from $\seq{b} \mi \seq{\0}$ and hence the claim follows from the exchange principle in field theory. So suppose that $a$ is not a root of any $F(X, b)$. Then $\phi(X, b)$ contains no $\VF$-sort equalities. If $\rv(F(X, b))$ occurs in $\phi(X, b)$ then for any $d \in \VF$ with $\vv(d - a)$ sufficiently large we have $\rv(F(a, b)) = \rv(F(d, b))$. So we see that $\phi(X, b)$ does not define a finite subset, contradiction.

For the second claim, let $\phi(X, s)$ be a quantifier-free formula
in disjunctive normal form that witnesses $t \in \acl(s)$. Clearly
we may assume that $\phi(X, s)$ does not contain any $\VF$-sort literal.
So $\phi(X, s)$ only contains $\RV$-literals. It is easily seen that the
inequalities cannot define nonempty finite subset and neither can
the disequalities. Therefore every irredundant disjunct of $\phi(X, s)$
has an equality conjunct. Since $t \notin \acl(\0)$, the claim follows again from the exchange principle in field theory.
\end{proof}

\begin{lem}\label{function:dim:1:range:decom}
Let $A$, $B \sub \VF$ and $f : A \fun B$ a definable surjective function. Then there are definable disjoint subsets $B_1$, $B_2 \sub Y$ such that
\begin{enumerate}
  \item $B_1 \cup B_2 = B$ and $B_1$ is finite,
  \item $f^{-1}(b)$ is infinite for each $b \in B_1$,
  \item the function $f \rest f^{-1}(B_2)$ is finite-to-one.
\end{enumerate}
\end{lem}
\begin{proof}
For each $b \in B$, if $f^{-1}(b)$ is infinite then, by compactness, there is an $a \in f^{-1}(b)$ such that $a \notin \acl(b)$. Since $b \in \seq{a} \sub \acl(a)$, by Lemma~\ref{exchange}, we must have $b \notin \acl(a) \mi \acl(\0)$ and hence $b \in \acl(\0)$. By compactness again there is a definable finite subset $B_1$ such that if $f^{-1}(b)$ is infinite then $b \in B_1$. Clearly we may adjust $B_1$ so that it contains exactly those $b \in B$ with $f^{-1}(b)$ infinite. So $B_1$ and $B_2 = B \mi B_1$ are as desired.
\end{proof}

We now turn to the study of balls. Let $\ga$ be an open ball and $\gb$ a ball. The following properties are easy to see.
\begin{enumerate}
 \item For any $c \in \VF$, the subset $\ga - c = \set{a - c: a \in \ga}$
 is an open ball. If $c \in \ga$ then $\vcr(\ga - c) = \infty$
 and $\rad(\ga - c) = \rad(\ga)$ and $\ga - c$ is a union of $\rv$-balls. If $c \notin \ga$ and $\vv(c) \leq \rad(\ga)$ then $\vcr(\ga - c) \leq \rad(\ga -
 c) = \rad(\ga)$. If $c \notin \ga$ and $\vv(c) > \rad(\ga)$
 then $\ga - c = \ga$.

 \item $0 \notin \ga$ if and only if $\ga$ is contained in an $\rv$-ball
 if and only if $\vcr(\ga) \neq \infty$ if and only if
 $\rad(\ga) \geq \vcr(\ga)$.

 \item The average of any finite set of elements in $\ga$ is in $\ga$ if and only if $\cha(\K) = 0$.

 \item For any $c_1$, $c_2 \in \VF$, $(\ga - c_1) \cap (\ga - c_2) \neq
 \0$ if and only if $\ga - c_1 = \ga - c_2$ if and only if
 $\vv(c_1 - c_2) > \rad(\ga)$.

 \item If $\ga \cap \gb = \0$ then $\vv(a - b) = \vv(a' - b')$ for all $a$, $a' \in \ga$, $b$, $b' \in \gb$ and the subset $\ga - \gb = \set{a - b : a \in \ga \text{ and } b \in \gb}$ is a ball that does not contain $0$. In fact, for any $a \in \ga$ and $b \in \gb$, either $\ga - \gb = \ga - b$ or $\ga - \gb = a - \gb$.

 \item Suppose $\ga \cap \gb = \0$. Let $\gc$ be the smallest closed ball that contains $\ga$. Clearly $\vcr(\gc) = \vcr(\ga)$ and $\rad(\gc) = \rad(\ga)$. If $\gb$ is a maximal open subball of $\gc$, that is, if $\gb$ is an open ball contained in $\gc$ with $\rad(\gb) = \rad(\gc)$, then $\ga - \gb$ is an $\rv$-ball $\rv^{-1}(t)$ with $\vv(t) = \rad(\ga)$. This means that the collection of maximal open subballs of $\gc$ admits a $\K$-affine structure.

 \item Let $f(x)$ be a polynomial with coefficients in $\VF$ and $d_1, \ldots, d_n$ the roots of $f(x)$. Suppose that $\ga$ is contained in an $\rv$-ball and does not contain any $d_i$. Then each $\ga - d_i$ is contained in an $\rv$-ball and hence $f(\ga)$ is contained in an $\rv$-ball, that is, $(\rv \circ f)(\ga)$ is a singleton.
\end{enumerate}
Similar properties are available if $\ga$ is a closed ball.

A ball $\gb$ may be represented by a triple $(a, b, d) \in \VF^3$,
where $a \in \gb$, $\vv(b)$ is the radius of $\gb$, and $d = 1$ if
$\gb$ is open and $d = 0$ if $\gb$ is closed. A set $\gB$ of balls
is a subset of $\VF^3$ of triples of this form such that if $(a,
b, d) \in \gB$ then for all $a' \in \VF$ with $\rv(a - a') \,
\Box_d \, b$, where $\Box_d$ is $>$ if $d = 1$ or $\geq$ if $d =
0$, there is a $b' \in \VF$ with $\vv(b) = \vv(b')$ such that
$(a', b', d) \in \gB$. Clearly two triples $(a, b , d)$, $(a', b',
d') \in \gB$ represent two different balls, which may or may not
be disjoint, if and only if either $(\vv(b),d) \neq (\vv(b'), d')$
or, in case that they are the same, $\rv(a - a') \, \Box_d \, b$
does not hold.

Let $\gB$ be a set of balls. We note the following terminological convention. The union of $\gB$, written as $\bigcup \gB$, is actually the collection of the elements in the first coordinate, that is,
\[
\bigcup \gB = \set{a : (a, b, d) \in \gB \text{ for some } b, d} \sub \VF.
\]
Sometimes the assertion $\bigcup \gB \sub A$ is simply written as $\gB \sub A$. We say that $\gB$ is finite if it contains finitely many distinct balls.
A subset of $\gB$ is always a set of balls in $\gB$. A function
$f$ of $\gB$ is always a function on the balls in $\gB$; that is,
$f$ is a relation between $\bigcup \gB$ and a set $A$ such that for
every $\gb \in \gB$ there is a unique $x \in A$ between which and
every $(a, b , d) \in \gb$ the relation holds. Notice that $f$ may
or may not be a function on the triples in $\gB$.

In a similar way a ball $\gb$ may be represented by a triple in
$\VF \times \RV^2$. This representation is sometimes more
convenient. Below we shall not distinguish these two
representations.

\begin{cor}
As imaginary definable subsets, $\Gamma$ is $o$-minimal and the set of maximal open balls contained in a closed ball is strongly
minimal.
\end{cor}

\begin{defn}
A subset $\gd$ of $\VF$ is a \emph{punctured} (\emph{open, closed,
$\rv$-}) \emph{ball} if $\gd = \gb \mi \bigcup_{i=1}^n \gh_i$, where
$\gb$ is an (open, closed, $\rv$-) ball, $\gh_i, \ldots, \gh_n$
are disjoint balls, and $\gh_i, \ldots, \gh_n \sub \gb$. Each
$\gh_i$ is a \emph{hole} of $\gd$. The \emph{radius} and the
\emph{valuative center} of $\gd$ are those of $\gb$. A subset $\gs$ of
$\VF$ is a \emph{simplex} if it is a finite union of disjoint balls and punctured balls of the same radius and the same valuative center, which are defined to be the \emph{radius} and the \emph{valuative center} of $\gs$ and are denoted by $\rad(\gs)$ and $\vcr(\gs)$.
\end{defn}

A special kind of simplex is called \emph{a thin annulus}: it is a
punctured closed ball $\gb$ with a single hole $\gh$ such that
$\gh$ is a maximal open ball contained in $\gb$. For example, an
element $\gamma \in \Gamma$ may be regarded as a thin annulus: it
is the punctured closed ball with radius $\gamma$ and valuative
center $\infty$ and the special maximal open ball containing $0$
removed.

\begin{defn}
Let $\gb_1, \ldots, \gb_n$ be the positive boolean components of a subset $A \sub \VF$. The \emph{positive closure} of $A$ is the set of the smallest closed balls $\set{\gc_1, \ldots, \gc_m}$ such that
each $\gc_i$ contains some $\gb_j$.
\end{defn}

Note that, if $A \sub \VF$ is definable from a set of parameters
then its positive closure is definable from the same set of
parameters.

\begin{rem}
By Theorem~\ref{c:min:acvf}, for any parametrically definable subset $A$ of $\VF$, there are disjoint balls and punctured balls $\ga_1, \ldots, \ga_l$ obtained from a unique set of balls $\gb_1, \ldots, \gb_n, \gh_1, \ldots, \gh_m$
such that $A = \bigcup_i \gb_i \mi \bigcup_j \gh_j$. If we group
$\ga_1, \ldots, \ga_l$ by their radii and valuative centers then
$A$ may also be regarded as the union of a unique set of disjoint parametrically definable simplexes. Each $\gb_i$ is a \emph{positive boolean component of $A$} and each $\gh_j$ is a \emph{negative boolean component of $A$}. The set of positive boolean components and the set of negative boolean components are both definable from the same parameters.
\end{rem}

\subsection{More structural properties and $b$-minimality}

\emph{For the rest of this section we shall assume that $\gC$ is of pure characteristic $0$}.

The following simple lemma is vital to the inductive arguments below.
It fails when $\cha{\K} > 0$.

\begin{lem}\label{average:0:rv:not:constant}
Let $c_1, \ldots, c_k \in \VF$ be distinct elements of the same
value $\alpha$ such that their average is $0$. Then for some $c_i
\neq c_j$ we have $\vv(c_i -c_j) = \alpha$ and hence $\rv$ is
not constant on the set $\set{c_1, \ldots, c_k}$.
\end{lem}
\begin{proof}
Suppose for contradiction that $\vv(c_i -c_j) > \alpha$ for all
$i \neq j$. Since $\cha{\K} = 0$ and $c_1 = -(c_2 + \ldots + c_k)$, we have
\[
\alpha = \vv (kc_1) = \vv ((k-1)c_1 - (c_2 + \ldots + c_k))
= \vv \left( \sum_{i =2}^k (c_1 - c_i) \right) > \alpha,
\]
contradiction.
\end{proof}

An important consequence of Lemma~\ref{average:0:rv:not:constant} is this:

\begin{lem}\label{finite:VF:project:RV}
Let $A$ be a definable finite subset of $\VF^n$. Then there is a definable injection $f : A \fun \RV^m$ for some $m$.
\end{lem}
\begin{proof}
We do double induction on $n$ and the number $k$ of elements in $A$. For $n = 1$, let $A = \set{c_1, \ldots, c_k} \sub \VF$ and $c$ the average of $A$. Then we may assume that $A = \set{c_1 -c, \ldots, c_k-c}$ and hence the average of $A$ is 0. Since every $\vv(c_i -c)$ is definable, by the inductive hypothesis we may further assume that $\vv$ is constant on $A$, say, $\vv(c_i) = \alpha$ for all $i$. By
Lemma~\ref{average:0:rv:not:constant}, $\rv$ is not constant on
$A$, that is, $1 < \abs{\rv(A)} \leq k$. So $1 \leq
\abs{\rv^{-1}(t)\cap A} < k$ for each $t \in \rv(A)$. By the inductive hypothesis, for a suitable number $m$, there is a $t$-definable injection
\[
f_t : \rv^{-1}(t) \cap A \fun \RV^m
\]
for each $t \in \rv(A)$. Then, by compactness, the function $f : A \fun \RV^{m+1}$ given by
\[
c_i \efun (\rv(c_i), f_{\rv(c_i)}(c_i))
\]
is definable and is as required.

Now suppose $n > 1$. Let $\pr_n(A)$ be the projection of $A$ to the last coordinate. For each $c \in \pr_n(A)$ let $\fib(A, c)$ be the fiber $\set{\lbar a : (\lbar a, c) \in A}$. By the inductive hypothesis, for a suitable number $m$, there is a definable injection $g: \pr_n(A) \fun \RV^m$ and, for each $c \in \pr_n(A)$, a $c$-definable injection $f_{c} : \fib(A, c) \fun \RV^m$. Then, by compactness, the function $f : A \fun \RV^{2m}$ given by
\[
(\lbar a, c) \efun (f_c(\lbar a), g(c))
\]
is definable and is as required.
\end{proof}

\begin{lem}\label{function:dim1:decom:RV}
Let $A$, $B \sub \VF$ and $f : A \fun B$ a definable surjective function. Then there is a definable function $P : A \fun \RV^m$ such that, for each $\lbar t \in \ran(P)$, $f \rest P^{-1}(\lbar t)$ is either constant or injective.
\end{lem}
\begin{proof}
Let $B_1$, $B_2$ be a partition of $B$ as given by Lemma~\ref{function:dim:1:range:decom}. By Lemma~\ref{finite:VF:project:RV}, there is an injection $B_1 \fun \RV^l$. The same holds for every $f^{-1}(b)$ with $b \in B_2$. So the lemma follows from compactness.
\end{proof}

\begin{lem}\label{function:rv:to:vf:finite:image}
Let $A$ be a definable subset of $\RV^m$ and $f: A \fun \VF^n$ a definable function. Then $f(A)$ is finite.
\end{lem}
\begin{proof}
We do induction on $n$. For the base case $n =1$, suppose for contradiction that $f(A)$ is infinite. By \cmin-minimality, $f(A)$ is a union of disjoint balls and punctured balls $\gb_1, \ldots, \gb_l$ such that $\rad \gb_i < \infty$ for
some $i$, say $\gb_1$. By QE, let $\phi$ be a
disjunction of conjunctions of literals that defines $f$. Since $f(A)$ is infinite, there is at least one disjunct in $\phi$, say $\phi^*$, that does not have an irredundant $\VF$-sort equality as a conjunct. Fix a $b \in \gb_1$ and a $\lbar t \in A$ such that
\begin{enumerate}
  \item the pair $(\lbar t, b)$ satisfies $\phi^*$,
  \item for any polynomial $G(X)$ occurring in $\phi^*$ in the form $\rv(G(X))$, $G(b) \neq 0$.
\end{enumerate}
We see that, for any $d \in \VF$ and any term $\rv(G(X))$ in $\phi^*$, if $\vv(d -b)$ is sufficiently large then $\rv(G(b)) = \rv(G(d))$. So there
is a $d \in \gb_1$ such that the pair $(\lbar t, d)$ also
satisfies $\phi^*$, which is a contradiction as $f$ is a function. In general, for $n > 1$, by the inductive hypothesis both $(\pr_1 \circ f) (A)$ and $(\pr_{>1} \circ f) (A)$ are finite, where $\pr_1 : \VF^n \fun \VF$ and  $\pr_{>1} : \VF^n \fun \VF^{n-1}$ are coordinate projections, hence $f(A)$ is finite.
\end{proof}

\begin{lem}\label{approx:roots}
Let $\gb$ be a ball contained in an $\rv$-ball $t$. Let $G_1(X), \ldots, G_n(X)$ be polynomials with coefficients in $S$. Suppose that $\gb$ does not contain any root of any $G_i(X)$ (hence $\rv$ is constant on every $G_i(\gb)$). If $\gb$ is a closed ball then there is a $d \in t \mi \gb$ such that $\rv(G_i(d)) = \rv(G_i(\gb))$ for every $i$. If $\gb$ is an open ball then there is a $d \in t \mi \gb$ such that $\vv(G_i(d)) = \vv(G_i(\gb))$ for every $i$.
\end{lem}
\begin{proof}
Since the argument is essentially the same for every $n$, for simplicity, we assume $n=1$ and the polynomial is written as $G(X)$. Let $a_1, \ldots, a_k$ be the roots of $G(X)$. Then there is a $d \in t \mi \gb$ such that, if $\gb$ is a closed ball then
\[
\vv(a_i - \gb) = \vv(a_i - d) < \vv(d - \gb) \leq \rad \gb
\]
for every $a_i$ and if $\gb$ is an open ball then
\[
\vv(a_i - \gb) = \vv(a_i - d) \leq \vv(d - \gb) \leq \rad \gb
\]
for every $a_i$. So, for such an element $d$: if $\gb$ is a closed ball then $\rv(a_i - \gb) = \rv(a_i - d)$ and hence $\rv(G(d)) = \rv(G(\gb))$; if $\gb$ is an open ball then at least $\vv(G(d)) = \vv(G(\gb))$.
\end{proof}

\begin{defn}
Let $\gB$ be a definable set of balls. If $\gB$ contains finitely many (open, closed, $\rv$-) balls, say, $\gb_1, \ldots, \gb_n$, then $\gB$ is an \emph{algebraic set of balls}, $\bigcup \gB$ is an \emph{algebraic union of balls}, and each $\gb_i$ is an \emph{algebraic} (\emph{open, closed, $\rv$-})
\emph{ball}. If there is a definable function $f: \gB \fun \VF$ such that
$f(\gb_i) \in \gb_i$ for every $\gb_i$ then we say that
$\gB$ has \emph{centers} and $f(\gB)$ is a \emph{set of centers of $\gB$}.
\end{defn}

\begin{lem}\label{effectiveness}
Let $\gB$ be an algebraic set of closed balls. Then $\gB$ has centers.
\end{lem}
\begin{proof}
Let $\gb_1, \ldots, \gb_n$ be the closed balls in $\gB$. Without loss of generality we may assume $\rad \gb_i < \infty$ and $0 \notin \gb_i$ for each $\gb_i$, that is, each $\gb_i$ is an infinite subset and is properly contained in an $\rv$-ball. Let $\phi(X)$ be a disjunction of conjunctions of literals that defines $\bigcup \gB$. Note that $\phi(X)$ must contain an irredundant $\RV$-sort literal. Let $F_j(X)$ enumerate all polynomials in $\VF[X]$ that occur in $\phi(X)$ in the form $\rv(F_j(X))$.

We claim that each $\gb_i$ contains a root of some $F_j(X)$. To see this, let $R$ be the finite set of all roots of all $F_j(X)$ and suppose for contradiction that $\gb_1$ is disjoint from $R$. Since every $\gb_i$ is a closed ball, there is an open ball $\ga$ that contains $\gb_1$ and is disjoint from every $\gb_i$ with $i > 1$. By the proof of Lemma~\ref{approx:roots}, we choose $\rad(\ga)$ so large that, for every $d \in \ga \mi \gb$, $\rv(F_j(d)) = \rv(F_j(\gb_1))$ for every $j$. Since $\gb_1$ is an infinite subset, there is a $b \in \gb_1$ such that $b$ satisfies a disjunct $\phi^*$ of $\phi$ and $\phi^*$ lacks $\VF$-sort equality. Then there is a $d \in \ga \mi \gb$ also satisfies $\phi^*$,
contradiction.

Let $b_i$ be the average of $\gb_i \cap R$. Then $b_i \in \gb_i$ and the function given by $\gb_i \fun b_i$ is as required.
\end{proof}

\begin{lem}\label{rv:effectiveness}
If $t \in \RV$ has a definable proper subset then it has definable center.
\end{lem}
\begin{proof}
Let $A$ be a definable proper subset of $t$. Let $\gb_1, \ldots, \gb_n$ be the positive boolean components of $A$ and $\gh_1, \ldots, \gh_m$ the negative boolean components of $A$. Since $A$ is a proper subset of $t$, at least one of
these balls is a proper subball of $t$ and hence its positive closure is also a proper subball of $t$. If we consider the set of the positive closures of these balls that are contained in $t$ then, by Lemma~\ref{effectiveness}, we obtain a definable finite subset of $t$ and hence, by taking the average, a definable
point in $t$.
\end{proof}

If the substructure $S$ does not contain excessive information from the $\RV$-sort, for example, if $S$ is $(\VF, \Gamma)$-generated, then it is also possible to have centers for algebraic sets of open balls (although this is not needed in this paper). To show this, we need the following observation. Suppose that $S$ is $\VF$-generated. Let $\lbar \gamma \in \Gamma^n$, $X, X_1, \ldots, X_n$ $\VF$-sort variables, $Y_1, \ldots, Y_n$ $\RV$-sort variables, and $\phi(X, \lbar Y)$ a quantifier-free formula. Suppose that for each $\lbar t \in \lbar \gamma$ the formula $\phi(X, \lbar t)$ defines the same subset $A \sub \VF$. For each $\lbar t \in \lbar \gamma$, clearly $A$ may also be defined by the formula
\[
\fa{\lbar X} (\vv(\lbar X) = \vrv(\lbar t) \limplies \phi(X, \rv(\lbar X))).
\]
Since $S$ is $\VF$-generated, as in the proof of Theorem~\ref{c:min:acvf}, $\phi(X, \rv(\lbar X))$ may be translated into an $\lan{v}$-formula $\psi(X, \lbar X)$ and hence $A$ may be defined by the $\lan{v}$-formula
\[
\fa{\lbar X} (\vv(\lbar X) = \lbar \gamma \limplies \psi(X, \lbar X)).
\]
In short, if $S$ is $(\VF, \Gamma)$-generated then any definable set in $\lan{RV}$ is also definable in $\lan{v}$ from the same parameters.

\begin{lem}\label{algebraic:balls:definable:centers}
Suppose that $S$ is $(\VF, \Gamma)$-generated. Let $\gB$ be an algebraic set of balls. Then $\gB$ has centers.
\end{lem}
\begin{proof}
Since the set of the closed balls in $\gB$ is definable, by Lemma~\ref{effectiveness}, we may assume that $\gB$ is an algebraic set of open balls, say, $\gb_1, \ldots, \gb_n$. As in Lemma~\ref{effectiveness}, we may also assume that each $\gb_i$ is contained in an $\rv$-ball (but perhaps not properly). Let $\phi(X)$ be a quantifier-free $\lan{v}$-formula that defines $\bigcup \gB$. Note that $\phi(X)$ must contain an irredundant $\Gamma$-sort literal. Now we may proceed exactly as in Lemma~\ref{effectiveness}, using the other part of Lemma~\ref{approx:roots}.
\end{proof}

\begin{cor}\label{aclS:model}
Suppose that $S$ is $\VF$-generated. Let $\acl(S)$ be the model-theoretic algabraic closure of $S$. If the value group $\Gamma(\acl(S))$ is nontrivial then $\acl(S)$ is a model of $\ACVF_S(0,0)$.
\end{cor}
\begin{proof}
We only need to show that any $t \in \RV(\acl(S))$ has a point in $\VF(\acl(S))$, which follows from Lemma~\ref{algebraic:balls:definable:centers}.
\end{proof}

\begin{defn}
Let $A \sub \VF^n \times \RV^m$ and
\[
\can(A) = \set{(\lbar a, \rv(\lbar a), \lbar t) : (\lbar a, \lbar t) \in A}
\sub \VF^n \times \RV^{n+m}.
\]
Clearly $A$ is bijective to $\can(A)$ in a canonical way. This bijection is called the \emph{canonical bijection} and is denoted by $\can$.
\end{defn}

\begin{conv}\label{convention:can:map}
In the discussion below it is very convenient to identify a definable subset $A$ with its canonical image $\can(A)$. Whether or not such an identification is made will always be clear in context. For example, in Definition~\ref{defn:special:bijection} below, it would not make sense without substituting $\can(A)$ for $A$.
\end{conv}

For any definable subset $A$, both the subset of $A$ that contains
all the $\rv$-polydiscs contained in $X$ and the superset of $A$
that contains all the $\rv$-polydiscs with nonempty intersection
with $A$ are definable.

\begin{defn}\label{defn:RV:product}
For any subset $U \sub \VF^n \times \RV^m$, the \emph{$\RV$-hull}
of $U$, denoted by $\RVH(U)$, is the subset $\bigcup\set{\rv^{-1}(\lbar t) \times \set{\lbar s} : (\lbar t, \lbar s) \in \rv(U)}$. If $U = \RVH(U)$, that is, if $U$ is a union of $\rv$-polydiscs, then we say that $U$ is an
\emph{$\RV$-pullback}.
\end{defn}

\begin{defn}\label{defn:special:bijection}
Let $A \sub \VF \times \RV^m$ and $C \sub \RVH(A)$ an $\RV$-pullback. Let $\pr_{>1}(C \cap A)$ be the projection of $C \cap A$ to the coordinates other than the first one. Let $\lambda: \pr_{>1}(C \cap A) \fun \VF$ be a function
such that $(\lambda(\lbar t), \lbar t) \in C$ for every $\lbar t \in \pr_{>1}(C \cap A)$. Let
\begin{gather*}
C^{\sharp} = \bigcup_{(t_1, \lbar t_1) \in \pr_{>1} C} \bigl( \bigl(\bigcup \set{\rv^{-1}(t): \vrv(t) > \vrv(t_1)}\bigr) \times \set{(t_1, \lbar t_1)} \bigr),\\
\RVH(A)^{\sharp} = C^{\sharp} \uplus (\RVH(A) \mi C).
\end{gather*}
The \emph{centripetal transformation $\eta : A \fun \RVH(A)^{\sharp}$ with respect to $\lambda$} is defined by
\[
\begin{cases}
  \eta (a, \lbar t) = (a - \lambda(\lbar t), \lbar t), & \text{on } C \cap A,\\
  \eta = \id, & \text{on } A \mi C.
\end{cases}
\]
Note that $\eta$ is
injective. The inverse of $\eta$ is naturally called the \emph{centrifugal transformation with respect to $\lambda$}. The function $\lambda$ is called a \emph{focus map of $A$}. The $\RV$-pullback $C$ is
called the \emph{locus} of $\lambda$. A \emph{special bijection}
$T$ is an alternating composition of centripetal transformations
and the canonical bijection. The \emph{length} of a special
bijection $T$, denoted by $\lh T$, is the number of centripetal
transformations in $T$. The image $T(A)$ is sometimes denoted as $A^{\sharp}$.
\end{defn}

Clearly if $A$ is an $\RV$-pullback and $T$ is a special bijection on $A$ then $T(A)$ is an $\RV$-pullback. Notice that a special bijection $T$ on $A$ is
definable if $A$ and all the focus maps involved are definable. Since we are only interested in definable subsets and definable functions on them, we further require a special bijection to be definable.

Special transformations are an important ingredient in the Hrushovski-Kazhdan integrations theory~\cite{hrushovski:kazhdan:integration:vf}. Definition~\ref{defn:special:bijection} is a specialized version that only involves one $\VF$-coordinate. Its general version (in all dimensions) will be studied in a sequel (also see~\cite[Section~7]{yin:hk:part:1}). Here we give a couple of examples.

\begin{exam}\label{example:special:bi}
Let $\gb \sub \VF$ be a definable open ball properly contained in a $t \in \RV$. By Convention~\ref{convention:can:map}, $\gb$ is identified with the subset $\gb \times \set{t}$. By Lemma~\ref{rv:effectiveness}, $t$ contains a definable element $a$, which may or may not be in $\gb$. Let $\lambda$ be the focus map $t \efun a$. Then the centripetal transformation on $\gb$ with respect to $\lambda$ is given by $(b, t) \efun (b - a, t)$.

Let $t_1, \ldots, t_n \in \RV$, $A_i \sub t_i$ a finite subset for each $i$, and $A = \bigcup_i (A_i \times {t_i})$. Suppose that $A$ is definable. Let $a_i$ be the average of $A_i$ and $\lambda$ the focus map given by $t_i \efun a_i$. Then the centripetal transformation $\eta$ on $A$ with respect to $\lambda$ is given by $(a, t_i) \efun (a - a_i, t_i)$. The special transformation $\can \circ \eta$ on $A$ is given by $(a, t_i) \efun (a - a_i, \rv(a - a_i), t_i)$. Notice that for any $i$, by Lemma~\ref{average:0:rv:not:constant}, $\rv$ is not constant on the subset $A_i - a_i$ and hence for any $s \in \RV$ the size of $\set{a : (a, s, t_i) \in \can \circ \eta(A)}$ is strictly smaller than the size of $A_i$. This phenomenon is the basis of the inductive arguments below.
\end{exam}

\begin{defn}
A definable subset $A$ is a \emph{deformed $\RV$-pullback} if there is a
special bijection $T$ such that $T(A)$ is an $\RV$-pullback.
\end{defn}

\begin{rem}\label{preimage:of:special}
Let $A$ be a deformed $\RV$-pullback and $T : A \fun U$ a special bijection that witnesses this. By a routine induction we see that
\begin{enumerate}
  \item if $(0, \infty, \lbar t) \in U$ then $T^{-1}(0, \infty, \lbar t)$ is a singleton,
  \item if $\rv^{-1}(s) \times \set{(s, \lbar t)} \in U$ then $T^{-1}(\rv^{-1}(s) \times \set{(s, \lbar t)})$ is an open polydisc.
\end{enumerate}
\end{rem}

Here is our key lemma:

\begin{lem}\label{simplex:with:hole:rvproduct}
Every definable subset $A \sub \VF \times \RV^m$ is a deformed $\RV$-pullback.
\end{lem}
\begin{proof}
By compactness, it is enough to show that, for every $(a, \lbar t) \in A$, there is a special bijection $T$ on $A$ such that $T(a, \lbar t)$ is contained in an $\rv$-polydisc $\gp \sub T(A)$. Fix an $(a, \lbar t) = (a, t_1, \ldots, t_m) \in A$. Let $B$ be the union of the $\rv$-polydiscs contained in $A$, which is a definable $\RV$-pullback. If $(a, \lbar t) \in B$ then the canonical bijection is as required. So, without loss of generality, we may assume that $B = \0$. By Convention~\ref{convention:can:map}, the canonical bijection has been applied to $A$ and hence the $\lbar t$-definable subset $\fib(A, \lbar t) = \set{b : (b, \lbar t) \in A}$ is properly contained in the $\rv$-ball $\rv^{-1}(t_1)$.

By \cmin-minimality, $\fib(A, \lbar t)$ is a disjoint union of $\lbar t$-definable simplexes. Let $\gs$ be the simplex that contains
$a$. Let $\gb_1, \ldots, \gb_l$, $\gh_1, \ldots, \gh_n$
be the boolean components of $\gs$, where each $\gb_i$ is positive
and each $\gh_i$ is negative. The proof now proceeds by induction
on $n$.

For the base case $n = 0$, $\gs$ is a disjoint union of balls $\gb_1, \ldots, \gb_l$ of the same radius and valuative center. Without loss of generality, we may assume $a \in \gb_1$. Let $\set{\gc_1, \ldots, \gc_k}$ be the positive closure of $\gs$. Note that this closure is also $\lbar t$-definable. We now
start a secondary induction on $k$. For the base case $k =1$, by
Lemma~\ref{effectiveness}, there is a $\lbar t$-definable
point $c \in \gc_1$. Clearly $\gc_1 - c \sub \rv^{-1}(t_1) - c$
is a union of $\rv$-balls. We see that there is a definable $C \sub \RVH(A)$ and a focus map $\lambda: \pr_{>1}(C \cap A) \fun \VF$ such that $\lambda(\lbar t) = c$. Then the centripetal transformation $\eta$ with respect to $\lambda$ is
as desired. For the inductive step of the secondary induction, by
Lemma~\ref{effectiveness} again, there is a $\lbar t$-definable set of centers $\set{c_1, \ldots, c_k}$ with $c_i \in \gc_i$. Let $c$ be the average of $c_1, \ldots, c_k$. Let $\lambda$, $\eta$ be as above such that $\lambda(\lbar t) = c$. If $c \in \gc_1$ then, as above, the centripetal transformation
$\eta$ with respect to $\lambda$ is as desired. So suppose $c \notin \gc_1$. Note that if $\vv$ is not constant on the set $\set{c_1 - c, \ldots, c_k - c}$ then $\rv$ is not constant on it and if $\vv$ is constant on it then, by
Lemma~\ref{average:0:rv:not:constant}, $\rv$ is still not constant
on it. Consider the special bijection $T = \can \circ \eta$. We have
\[
T(a, \lbar t) = (a - c, r, \lbar t) \in T(A),
\]
where $r = \rv(a-c)$. Observe that the positive closure of the $(r, \lbar t)$-definable
subset
\[
\fib(T(A), (r, \lbar t)) = \set{b : (b, r, \lbar t) \in T(A)}
\]
is a proper subset of the set $\set{\gc_1 - c, \ldots, \gc_k - c}$ of closed balls. Hence, by the inductive hypothesis, there is a special bijection $T'$ on
$T(A)$ such that $T'(a - c, r, \lbar t)$ is contained in an
$\rv$-polydisc $\gp \sub T' \circ T(A)$. So $T' \circ T$ is as required. This completes the base case $n=0$.

We proceed to the inductive step. Note that, since $\gb_1, \ldots,
\gb_l$ are of the same radius and are pairwise disjoint, the holes $\gh_1, \ldots, \gh_n$ are also pairwise disjoint. Without loss of generality we may also assume
that all the holes $\gh_1, \ldots, \gh_n$ are of the same radius.
Let $\set{\gc_1, \ldots, \gc_k}$ be the positive closure of
$\bigcup_i \gh_i$. The secondary induction on $k$ above may be
carried out here almost verbatim with respect to $\set{\gc_1, \ldots, \gc_k}$: the point is, in the inductive step, after applying the special bijection $T$, the number of holes in the fiber that contains $T(a, \lbar t)$ decreases
and hence the inductive hypothesis may be applied.
\end{proof}

\begin{cor}\label{b:minimal:con:13}
Let $A$, $B \sub \VF$ and $f : A \fun B$ a definable surjective function. Then there is a definable function $P : A \fun \RV^m$ such that, for each $\lbar t \in \ran(P)$, $P^{-1}(\lbar t)$ is an open ball or a point and $f \rest P^{-1}(\lbar t)$ is either
constant or injective.
\end{cor}
\begin{proof}
Let $P_0 : A \fun \RV^l$ be a function as given by Lemma~\ref{function:dim1:decom:RV}. Applying Lemma~\ref{simplex:with:hole:rvproduct} to each fiber $P_0^{-1}(\lbar t)$ we get a $\lbar t$-definable special bijection $T_{\lbar t}$. Let $P_{\lbar t}$ be the composition of $T_{\lbar t}$ and the projection to the $\RV$-coordinates. Then, by Remark~\ref{preimage:of:special}, the function $P : A \fun \RV^m$ given by
\[
a \efun (P_0(a), P_{P_0(a)}(a))
\]
is as required.
\end{proof}

\begin{thm}\label{b:min:acvf}
The theory $\ACVF_S(0,0)$ is \bmin-minimal.
\end{thm}
\begin{proof}
For the three conditions in Definition~\ref{def:b:min}, (b2) is given by Lemma~\ref{function:rv:to:vf:finite:image} and (b1), (b3) follow from Corollary~\ref{b:minimal:con:13}.
\end{proof}

\section{Minimality in $\ACVF^{\dag}$ and $\ACVF^{\ddag}$}

The main object of this section is to compare $\ACVF^{\dag}$ and $\ACVF^{\ddag}$ in terms of the geometry of definable sets, or more precisely, minimality conditions. Note that $\ACVF^{\dag}$ and $\ACVF^{\ddag}$ are clearly not \cmin-minimal. However, they are both $b$-minimal, as shown below.

Let $\gC$ be a sufficiently saturated model of $\ACVF^{\dag}$ or $\ACVF^{\ddag}$, depending on the context. We fix a small substructure $S \sub \gC$ and work in $\ACVF^{\dag}_S$ or $\ACVF^{\ddag}_S$. For simplicity we shall still refer to the language of $\ACVF^{\dag}_S$ and $\ACVF^{\ddag}_S$ as $\lan{RV}^{\dag}$.

\begin{defn}
Let $\tau$ be an $\lan{RV}^{\dag}$-term. The \emph{complexity $\abs{\tau} \in \N$} of $\tau$ is defined inductively as follows.
\begin{enumerate}
  \item If $\tau$ is an $\lan{RV}$-term then $\abs{\tau} = 0$.
  \item If $\tau$ is of the form $\sn(\sigma)$ and $\sigma$ contains the function $\rv$ then $\abs{\tau} = \abs{\sigma} + 1$.
  \item If $\sigma(X_1, \ldots, X_n)$ is an $\lan{RV}$-term with $X_1, \ldots, X_n$ occurring variables then $\abs{\sigma(\tau_1, \ldots, \tau_n)} = \max \{\abs{\tau_1}, \ldots, \abs{\tau_n}\}$.
\end{enumerate}
The \emph{complexity $\abs{\phi} \in \N$} of a formula $\phi$ is the maximal complexity of the terms occurring in $\phi$.
\end{defn}

\begin{lem}\label{rv:fun:vf:no:ball}
Let $f : \RV^m \fun \VF$ be a definable function. Then $\ran(f)$ does not contain any open ball.
\end{lem}
\begin{proof}
Let $\phi(X, \lbar Y)$ be a disjunction of conjunctions that defines $f$, where $X$ is a $\VF$-sort variable and $\lbar Y$ are $\RV$-sort variables. Let $\sigma_1(X, \lbar Y), \ldots, \sigma_k(X, \lbar Y)$ be all the distinct terms occurring in $\phi(X, \lbar Y)$ in the form $\rv(\sigma_i(X, \lbar Y))$ with $|\sigma_i(X, \lbar Y)| \leq 1$. Let $Z_1, \ldots, Z_k$ be $\RV$-sort variables and $\phi_1(X, \lbar Y, \lbar Z)$ the formula obtained from $\phi(X, \lbar Y)$ by replacing $\rv(\sigma_i(X, \lbar Y))$ with $Z_i$. Let $\psi(X, \lbar Y, \lbar Z)$ be the formula
\[
\phi_1(X, \lbar Y, \lbar Z) ) \wedge  \bigwedge_i \rv(\sigma_i(X, \lbar Y)) = Z_i,
\]
which defines a partial function $g : \RV^{m+k} \fun \VF$ such that $\ran(g) = \ran(f)$. Clearly $|\psi(X, \lbar Y, \lbar Z)| < |\phi(X, \lbar Y)|$. Repeating this procedure, we see that it is enough to prove the case $|\phi(X, \lbar Y)| \leq 1$.

So let $|\phi(X, \lbar Y)| \leq 1$. Suppose for contradiction that $\ran(f)$ contains an open ball. Let $F_i(X) \in \VF[X]$ enumerate all the polynomials occurring in $\phi(X, \lbar Y)$. Then there is an open ball $\gb \sub \ran(f)$ such that, for every $a$, $a' \in \gb$ and every $i$, $\rv(F_i(a)) = \rv(F_i(a')) \neq \infty$. Let $\phi_1(X, \lbar Y)$ be the formula obtained from $\phi(X, \lbar Y)$ by replacing $\rv(F_i(X))$ with $\rv(F_i(\gb))$. Let both $(\lbar t, a)$ and $(\lbar t', a')$ satisfy a disjunct $\phi^*_1(X, \lbar Y)$ of $\phi_1(X, \lbar Y)$. Note that we may choose $a$, $a'$ so that
\begin{enumerate}
  \item $a$ is transcendental over the field generated by $\sn(\RV)$ and the $\VF$-sort parameters occurring in $\phi(X, \lbar Y)$,
  \item $a'$ is arbitrarily close to $a$.
\end{enumerate}
Hence $\phi^*_1(X, \lbar Y)$ does not contain any $\VF$-sort equalities and $T(a, \lbar t) \neq 0$ for every term of the form $\rv(T(X, \lbar Y))$ occurring in $\phi^*_1(X, \lbar Y)$. Then, as in Lemma~\ref{function:rv:to:vf:finite:image}, $\phi^*_1(a', \lbar t)$ also holds, contradiction.
\end{proof}

\begin{rem}\label{RV:part}
For any subset $A \sub \VF^n$ defined by a formula $\phi(\lbar X)$, by a routine induction on $|\phi(\lbar X)|$, we see that there is a definable function $\pi : A \fun \RV^m$ and an $\lan{RV}$-formula $\phi^*(\lbar X, \lbar Y)$ such that $\pi^{-1}(\lbar t)$ is defined by the formula $\phi^*(\lbar X, \sn(\lbar t))$. For details see a more specialized version Lemma~\ref{complex:1} below.
\end{rem}

\begin{thm}
Both $\ACVF^{\dag}_S(0,0)$ and $\ACVF^{\ddag}_S(0,0)$ are $b$-minimal.
\end{thm}
\begin{proof}
Condition (b2) in Definition~\ref{def:b:min} follows from Lemma~\ref{rv:fun:vf:no:ball}. For (b1), let $A \sub \VF$ be definable and $\pi$ a function for $A$ as described in Remark~\ref{RV:part}. Since each $\pi^{-1}(\lbar t)$ is $\sn(\lbar t)$-definable in $\lan{RV}$, by Theorem~\ref{b:min:acvf}, there is an $\sn(\lbar t)$-definable function $P_{\lbar t} : A \fun \RV^m$ that makes (b1) hold for $\pi^{-1}(\lbar t)$. By compactness, the function $P : A \fun \RV^l$ given by
\[
a \efun (\pi(a), P_{\pi(a)}(a))
\]
is definable and makes (b1) hold for $A$. Applying same argument to any definable function $f : A \fun \VF$, (b3) also follows.
\end{proof}

For the rest of this section, unless indicated otherwise, we suppose that $\sn$ is a section of $\K$. We shall introduce a modified version of \cmin-minimality, called local minimality, which $\ACVF^{\ddag}$ satisfies but $\ACVF^{\dag}$ does not.

\begin{defn}\label{defn:local:c}
Let $K$ a valued field considered as a structure of some language $\LL$ and $\Gamma$ its value group ($\Gamma$ is somehow definable, possibly as an imaginary sort). Let $A \sub K^n$ and $p: A \fun \Gamma$ an definable function. We say that $p$ is a \emph{volumetric partition} of $A$ if $p$ is constant on $\go(\lbar a, p(\lbar a)) \cap A$ for any $\lbar a \in A$.
\end{defn}

Volumetric partitions are so named because of their role in the integration theory (see~\cite[Section~3]{yin:hk:part:3}). For example, the valuation $\vv : K \fun \Gamma \cup \set{\infty}$ is a volumetric partition.

\begin{defn}
Let $K$ and $\LL$ be as in Definition~\ref{defn:local:c}. A definable subset $A \sub K$ is \emph{locally \cmin-minimal} if there is a volumetric partition $p: A \fun \Gamma$ such that $\go(a, p(\lbar a)) \cap A$ is a boolean combination of balls for every $a \in A$. We say that $K$ is \emph{locally \cmin-minimal} if every definable subset of $K$ is locally \cmin-minimal. A theory $T$ in $\LL$ is \emph{locally \cmin-minimal} if it includes the axioms for valued fields and every model of $T$ is locally \cmin-minimal.
\end{defn}

Consider the $\lan{RV}^{\dag}$-formula $\rv(X) = 1 \wedge \sn(X - 1) - X -1 = 0$. In $\ACVF^{\dag}$ it defines the set $\set{1 + \sn(t) : \vrv(t) > 0}$, which is clearly not locally $C$-minimal. On the other hand, in $\ACVF^{\ddag}$ this formula defines the singleton $\set{1}$.

\begin{lem}\label{term:simplfied:II}
Let $\sn(\sigma(X))$ be a term with $\abs{\sn(\sigma(X))} = 1$, where $X$ is a $\VF$-sort variable. Then there is a volumetric partition $p : \VF \fun \Gamma$ such that for each $\go(a, p(a))$ one of the following possibilities occurs:
\begin{enumerate}
  \item $\sn(\sigma(a'))$ is not defined for any $a' \in \go(a, p(a))$;
  \item $\sn(\sigma(a')) = \sn(\sigma(a))$ for every $a' \in \go(a, p(a))$.
\end{enumerate}
\end{lem}
\begin{proof}
Since $\sn$ is only nontrivially defined on $\K^{\times}$, we may assume that the $\lan{RV}$-term $\sigma(X)$ is of the form $\sum_{i=1}^k (\rv(F_{i}(X)) \cdot r_{i})$. Fix an $a \in \VF$. If $F_{i}(a) = 0$ for some $i$ then there is an $a$-definable $\gamma_{a} \in \Gamma$ such that, for every $a' \in \go(a, \gamma_{a})$, $\vv(F_{i}(a')) > - \vrv(r_i)$. Hence, on $\go(a, \gamma_{a})$, if $k > 1$ then (1) occurs and if $k =1$ then (2) occurs. If $F_{i}(a) \neq 0$ for all $i$ then there is an $a$-definable $\gamma_{a} \in \Gamma$ with the following property: $\gamma_{a}$ is the least value such that $\rv(F_i(a')) = \rv(F_i(a))$ for every $i$ and every $a' \in \go(a, \gamma_{a})$. It exists because $\Gamma$ is $o$-minimal. Note that if $a' \in \go(a, \gamma_{a})$ then $\gamma_{a} = \gamma_{a'}$. Therefore, on $\go(a, \gamma_{a})$, if $\sigma(a) = 0$ then (1) occurs and if $\sigma(a) \neq 0$ then (2) occurs.

Without loss of generality we may assume $\gamma_{a} \geq \vv(a)$. We construct a volumetric partition $p : \VF \fun \Gamma$ as follows. Let $A_1$ be the set of zeros of $F_{i}(X)$. Then there is a definable $\beta \in \Gamma$ such that $\beta \geq \gamma_a$ for all $a \in A_1$ and $\go(a_1, \beta) \cap \go(a_2, \beta) = \0$ for all $a_1$, $a_2 \in A_1$. Let $p_1 : \bigcup_{a \in A_1} \go(a, \beta) \fun \set{\beta}$ be the constant function. Let $A_2 = \VF \mi \bigcup_{a \in A_1} \go(a, \beta)$. For every $a \in A_2$ let $\beta_{a}$ be the least value such that $\go(a, \beta_{a}) \sub A_2$. These exist because $\Gamma$ is $o$-minimal. Let $p_2 : A_2 \fun \Gamma$ be such that $p_2(a) = \max\set{\gamma_{a}, \beta_{a}}$. Then $p = p_1 \cup p_2$ is as desired.
\end{proof}

This lemma is the key to showing local \cmin-minimality of $\ACVF^{\ddag}_S$, which fails in $\ACVF^{\dag}_S$. A more complicated version of it does hold in $\ACVF^{\dag}_S$:

\begin{lem}\label{term:simplfied}
Suppose that $\sn$ is a section of $\RV$. Let $\sn(\sigma(X))$ be a term with $\abs{\sn(\sigma(X))} = 1$. Then there is a volumetric partition $p : \VF \fun \Gamma$ such that for each $\go(a, p(a))$ one of the following possibilities occurs:
\begin{enumerate}
  \item $\sn(\sigma(a'))$ is not defined for every $a' \in \go(a, p(a))$;
  \item $\sn(\sigma(a')) = \sn(\sigma(a))$ for every $a' \in \go(a, p(a))$;
  \item there is a natural number $l$ (not depending on $a$), an element $c \in \go(a, p(a))$ (depending on $\go(a, p(a))$ rather than $a$), and an element $b \in \VF$ (depending on $c$ and hence on $\go(a, p(a))$) such that $\sn(\sigma(a')) = b \sn(\rv(a' - c))^l$ for every $a' \in \go(a, p(a))$.
\end{enumerate}
\end{lem}
\begin{proof}
The argument is very similar to that for Lemma~\ref{term:simplfied:II}, although here there is one more possibility. So we shall not spell out all the details when there is no danger of confusion.

The $\lan{RV}$-term $\sigma(X)$ is of the form $\rv(F(X)) \cdot r \cdot T(X)$, where $T(X)$ is a $\K$-term of the form $\sum_{i=1}^k (\rv(F_{i}(X)) \cdot r_{i})$ with $k > 1$ (if $k = 1$ then by convention $T(X)$ is the constant $1$). Fix an $a \in \VF$.

First suppose $k =1$. If $F(a) = 0$ then there is a definable $\gamma_a \in \Gamma$, an $a$-definable $d \in \VF$, and a natural number $l$ such that
\begin{itemize}
  \item there is only one root of $F(X)$ contained in $\go(a, \gamma_a)$, namely $a$,
  \item $\rv(F(a')) = \rv(d) \rv(a' - a)^l$ for every $a' \in \go(a, \gamma_a)$.
\end{itemize}
Thus (3) occurs with $c = a$. If $F(a) \neq 0$ then there is a definable $\gamma_a \in \Gamma$ such that $\rv(F(a')) = \rv(F(a))$ for every $a' \in \go(a, \gamma_a)$ and hence (2) occurs.

For the case $k > 1$ we can use the corresponding part in the proof of Lemma~\ref{term:simplfied:II}, noting that if $T(a) = 0$ then (1) occurs and if $T(a) \neq 0$ then we are back in the case $k = 1$.

The construction of a volumetric partition $p : \VF \fun \Gamma$ is more or less as in the proof of Lemma~\ref{term:simplfied:II}.
\end{proof}

\begin{lem}\label{complex:1}
Let $\phi(X)$ be a quantifier-free formula and $A \sub \VF$ the subset defined by it. Then there is a volumetric partition $p : \VF \fun \Gamma$ and a definable function $\pi : \VF \fun \RV^m$ such that
\begin{enumerate}
  \item $\pi$ is constant on every $\go(a, p(a))$,
  \item every intersection $\pi^{-1}(\lbar t) \cap A$ is $\lan{RV}$-definable with the parameters $\sn(\lbar t)$.
\end{enumerate}
\end{lem}
\begin{proof}
Let $\sigma_1(X), \ldots, \sigma_k(X)$ be all the distinct $\lan{RV}$-terms occurring in $\phi(X)$ in the form $\sn(\sigma_i(X))$. Let $p_0 : \VF \fun \Gamma$ be a volumetric partition that makes Lemma~\ref{term:simplfied:II} hold for every $\sn(\sigma_i(X))$. Let $\hat{\phi}(X, X_1, \ldots, X_k)$ be the formula obtained from $\phi(X)$ by replacing $\sn(\sigma_i(X))$ with a $\VF$-sort variable $X_i$. For each $a \in \VF$, if some $\sn(\sigma_i(a))$ is not defined then set
\[
\pi_0(a) = (\infty, \ldots, \infty) \in \RV^{k+1},
\]
otherwise set
\[
\pi_0(a) = (\sigma_1(a), \ldots, \sigma_k(a), 1) \in \RV^{k+1}.
\]
Clearly the function $\pi_0 : \VF \fun \RV^{k+1}$ is constant on every $\go(a, p_0(a))$. This means that each $\pi_0^{-1}(\lbar t)$ is a union of balls of the form $\go(a, p_0(a))$. Without loss of generality we may assume $\pi_0(\VF) \sub \RV^k \times \set{1}$. For every $\lbar t \in \pi_0(\VF)$ let
\[
\lbar a_{\lbar t} = (a_1, \ldots, a_k) = (\sn(t_1), \ldots, \sn(t_k)).
\]
Then the intersection $\pi_0^{-1}(\lbar t) \cap A$ is defined with $\lbar a_{\lbar t}$ by the formula
\[
\hat{\phi}(X, \lbar a_{\lbar t}) \wedge \bigwedge_i \sigma_i(X) = \rv(a_i),
\]
which shall be called $\phi_0(X, \lbar a_{\lbar t})$.

We now proceed by induction on $|\phi(X)|$. For the base case $|\phi(X)| \leq 1$, we see that $\phi_0(X, \lbar a_{\lbar t})$ is actually an $\lan{RV}$-formula and hence $p_0$, $\pi_0$ are as required. For the inductive step, since $|\phi_0(X, \lbar a_{\lbar t})| < |\phi(X)|$, we may apply the inductive hypothesis to $\phi_0(X, \lbar a_{\lbar t})$ to obtain two $\lbar t$-definable functions $p_{\lbar t}$ and $\pi_{\lbar t}$ on $\pi_0^{-1}(\lbar t)$ that satisfy the required conditions. Let $p : \VF \fun \Gamma$ be the function given by
\[
a \efun \max \set{p_0(a), p_{\pi_0(a)}(a)}
\]
and $\pi : \VF \fun \RV^l$ the function given by
\[
a \efun (\pi_0(a), \pi_{\pi_0(a)}(a)).
\]
By compactness these two functions are definable and hence are as required.
\end{proof}

\begin{thm}
The theory $\ACVF^{\ddag}_S$ is locally $C$-minimal.
\end{thm}
\begin{proof}
Let $A \sub \VF$ be definable and $p$, $\pi$ two functions as given by Lemma~\ref{complex:1} for $A$. Every $\pi^{-1}(\lbar t) \cap A$ is parametrically $\lan{RV}$-definable and hence, by Theorem~\ref{c:min:acvf}, is a boolean combination of balls. Since each $\go(a, p(a))$ is contained in some $\pi^{-1}(\lbar t)$, clearly $\go(a, p(a)) \cap A$ is also a boolean combination of balls.
\end{proof}

\end{document}